\documentclass[10pt]{article}
\usepackage{latexsym,amsfonts,amsmath,amssymb,array,tabularx,theorem}
\usepackage{graphicx}
\usepackage{caption}
\usepackage{subcaption}
\usepackage{comment}
\usepackage[normalem]{ulem}
\usepackage[pdfpagelabels]{hyperref}
\usepackage{bbm}
\usepackage{authblk}
\usepackage{bm}
\usepackage{mathtools} 
\usepackage{dsfont} 
\usepackage{algpseudocode}
\usepackage[section]{algorithm}

\usepackage[normalem]{ulem}

\usepackage{graphicx}
\usepackage{epstopdf}
\usepackage[all]{xy}
\DeclareFontFamily{OT1}{pzc}{}
\DeclareFontShape{OT1}{pzc}{m}{it}{<-> s * [1.10] pzcmi7t}{}
\DeclareMathAlphabet{\mathpzc}{OT1}{pzc}{m}{it}

\newcommand{\VI}{\boldsymbol{I}}
\newcommand{\VL}{{\boldsymbol{L}}}
\newcommand{\VA}{{\boldsymbol{A}}}
\newcommand{\VB}{{\boldsymbol{B}}}
\newcommand{\VC}{{\boldsymbol{C}}}
\newcommand{\LS}{{\mathrm{L}\mathrm{S}}}

\newcommand{\bsigma}{{\boldsymbol{\sigma}}}
\newcommand{\btau}{{\boldsymbol{\tau}}}
\newcommand{\bg}{{\boldsymbol{g}}}
\usepackage{float}
\newtheorem{theorem}{Theorem}[section]
\newtheorem{lemma}[theorem]{Lemma}

\newtheorem{proposition}[theorem]{Proposition}
\newtheorem{remark}[theorem]{Remark}
\newtheorem{definition}[theorem]{Definition}

\newenvironment{proof}{\begin{trivlist}
    \item[\hskip\labelsep{\bf Proof.}]}{$\hfill\Box$\end{trivlist}}

{\theoremstyle{plain} \theorembodyfont{\rmfamily}
}

\numberwithin{equation}{section}
\numberwithin{figure}{section}

\allowdisplaybreaks[1]

\newcommand{\bsa}{{\boldsymbol{a}}}

\newcommand{\bstau}{{\boldsymbol{\tau}}}

\newcommand{\bszero}{{\boldsymbol{0}}}
\newcommand{\bsb}{{\boldsymbol{b}}}

\newcommand{\bsf}{{\boldsymbol{f}}}

\newcommand{\bsp}{{\boldsymbol{p}}}
\newcommand{\bsu}{{\boldsymbol{u}}}
\newcommand{\bsv}{{\boldsymbol{v}}}
\newcommand{\bsw}{{\boldsymbol{w}}}

\newcommand{\bsq}{{\boldsymbol{q}}}

\newcommand{\bsz}{{\boldsymbol{z}}}

\newcommand{\bseps}{\boldsymbol\varepsilon}

\newcommand{\bsnul}{{\boldsymbol{0}}}
\newcommand{\bsvarrho}{{\boldsymbol{\varrho}}}

\newcommand{\bbR}{\mathbb{R}}
\newcommand{\bbD}{\mathbb{D}}
\newcommand{\bbT}{\mathbb{T}}
\newcommand{\bbL}{\mathbb{L}}
\newcommand{\bbV}{\mathbb{V}}

\newcommand{\bbX}{\mathbb{X}}
\newcommand{\R}{\mathbb{R}}

\newcommand{\N}{\mathbb{N}}

\newcommand{\bbP}{\mathbb{P}}
\newcommand{\calA}{\mathcal{A}}

\newcommand{\calF}{\mathcal{F}}

\newcommand{\calC}{\mathcal{C}}
\newcommand{\calL}{\mathcal{L}}

\newcommand{\calI}{\mathcal{I}}
\newcommand{\calM}{\mathcal{M}}

\newcommand{\calNN}{\mathcal{NN}}

\newcommand{\calT}{\mathcal{T}}

\newcommand{\calE}{\mathcal{E}}

\newcommand{\calV}{\mathcal{V}}
\newcommand{\calU}{\mathcal{U}}

\newcommand{\mask}[1]{{}}
\usepackage[usenames]{color}
\definecolor{dkgreen}{rgb}{0.0, 0.5, 0.0}
\definecolor{orange}{rgb}{1,0.5,0}
\definecolor{turquoise}{rgb}{0,0.7,0.7}

\newcommand{\norm}[2][]{\| #2 \|_{#1}}

\newcommand{\snorm}[2][]{| #2 |_{#1}}
\newcommand{\snormc}[2][]{\left| #2 \right|_{#1}}
\newcommand{\setc}[2]{\left\{#1\, :\,#2\right\}}
\newcommand{\set}[2]{\{#1\,:\,#2\}}

\newcommand{\domain}{{\mathrm{D}}}
\newcommand{\spatialdomain}{\mathrm{G}}
\newcommand{\Sop}[1]{\operatorname{S}^{\rm 1}_{#1}}

\newcommand{\Szp}[1]{\operatorname{S}^{\rm 0}_{#1}}
\newcommand{\So}{\operatorname{S}_1}
\newcommand{\RT}{\operatorname{RT}_0}
\newcommand{\Sz}{\operatorname{S}_0}
\DeclareMathOperator{\trace}{tr}
\DeclareMathOperator{\divv}{div}
\newcommand{\Ne}{\operatorname{N}_{0}}

\DeclareMathOperator{\curl}{curl}
\newcommand{\Ned}{N\'ed\'elec }

\def\cA{{\cal A}}

\def\cT{{\cal T}}

\newcommand{\be}{\begin{equation}}
\newcommand{\ee}{\end{equation}}
\newcommand{\bea}{\begin{eqnarray}}
\newcommand{\eea}{\end{eqnarray}}
\newcommand{\beas}{\begin{eqnarray*}}
\newcommand{\eeas}{\end{eqnarray*}}

\DeclareMathOperator*{\essinf}{ess\,inf}

\newcommand{\relu}{{\rho}}
\newcommand{\heavi}{{\sigma}}

\newcommand{\realiz}[1]{{\rm R}(#1)} 
\newcommand{\depth}{L}
\newcommand{\size}{M}
\newcommand{\Parallel}[1]{{\rm P}(#1)} 

\newcommand{\sconc}{\odot}

\newcommand{\sizefirst}{\size_{\operatorname{in}}}
\newcommand{\sizelast}{\size_{\operatorname{out}}}

\DeclareMathOperator{\conv}{conv}
\DeclareMathOperator{\interior}{int}

\newcommand{\ths}{\theta^*}
\newcommand{\thh}{\hat\theta}

\newcommand{\bsn}{{\bm n}}
\newcommand{\Id}{{\rm Id}}

\newcommand{\control}{X}
\newcommand{\ad}{\mathrm{ad}}

\algnewcommand\algorithmicinput{\textbf{Input:}}
\algnewcommand\Input{\item[\algorithmicinput]}
\algnewcommand\algorithmicoutput{\textbf{Output:}}
\algnewcommand\Output{\item[\algorithmicoutput]}
\algnewcommand\algorithmicloopstart{\textbf{Loop:}}
\algnewcommand\Loopstart{\item[\algorithmicloopstart]}

\newcommand{\Solve}{{\tt Solve}}
\newcommand{\Estimate}{{\tt Estimate}}
\newcommand{\Mark}{{\tt Mark}}
\newcommand{\Refine}{{\tt Refine}}

\begin{document}

\bibliographystyle{abbrv}
\title{First Order System Least Squares Neural Networks}
\author{Joost A. A. Opschoor\thanks{ASML and Seminar for Applied Mathematics, ETH Z\"{u}rich, R\"{a}mistrasse 101, CH--8092 Z\"urich, Switzerland, \texttt{joost.opschoor@sam.math.ethz.ch}},  Philipp C. Petersen \thanks{Faculty of Mathematics and Research Network Data Science, University of Vienna, Kolingasse 14-16, 1090 Vienna, \texttt{philipp.petersen@univie.ac.at}},  Christoph Schwab \thanks{ Seminar for Applied Mathematics, ETH Z\"{u}rich, R\"{a}mistrasse 101, CH--8092 Z\"urich, Switzerland, \texttt{christoph.schwab@sam.math.ethz.ch}}}
%

%
\maketitle
\date{}
\begin{abstract}
We introduce a conceptual framework for numerically solving 
linear elliptic, parabolic, and hyperbolic PDEs 
on bounded, polytopal domains in euclidean spaces by deep neural networks. 
The PDEs are recast as minimization of a 
least-squares (LSQ for short) residual of an 
equivalent, well-posed first-order system,
over parametric families of deep neural networks. 
The associated LSQ residual is 
a) equal or proportional to a weak residual of the PDE, 
b) additive in terms of contributions from localized subnetworks,
   indicating locally ``out-of-equilibrium'' of neural networks 
   with respect to the PDE residual,
c) serves as numerical loss function for neural network training,
and 
d) constitutes, even with incomplete training, a 
computable, (quasi-)optimal numerical error estimator 
in the context of adaptive LSQ finite element methods.
In addition, 
an adaptive neural network growth strategy is proposed which,
assuming exact numerical minimization of the LSQ loss functional, 
yields sequences of neural networks with 
realizations that converge rate-optimally to the 
exact solution of the first order system LSQ formulation.
\end{abstract}



\section{Introduction}
\label{sec:Intro}
%
Numerous recently proposed algorithms for deep neural network (NN) training 
in connection with partial differential equations (PDEs) in physical domains
$\domain$
are derived from energies in variational principles
(e.g. \cite{EYuDeepRitz, AinsDong, ainsworth2024extended}) 
even for linear, deterministic and well posed PDEs.
Loss functions can be based on classical- (i.e. pointwise), 
weak- and variational-formulations of the PDE and corresponding residuals,
see e.g. the overview in \cite[Section 2.3]{DM2024}.
Importantly, adopting ``convenient'' and ``intuitive'' loss functions
may entail implicit unrealistic assumptions on solution regularity, whereas
mathematically correct loss functions may entail the necessity 
of numerical evaluation of ``inconvenient'', nonlocal Sobolev norms of
NN residuals in the PDE of interest. 
Here, we indicate advantages of NN approximations
based on a first order system least squares (FoSLS) reformulation
of the underlying PDE. 
As we show, FoSLS variational formulations 
of PDEs lead to $L^2(\domain)$-NN residuals which, for NN emulations 
of spaces in the DeRham complex \eqref{derham}, \eqref{discderham} below, 
are easily numerically evaluated
exactly (up to unavoidable data oscillation errors):
based on \cite{LODSZ22_991}, 
we present compatible, structure preserving NN emulations 
of FoSLS methods for a wide range of boundary- and initial-boundary
value problems of PDEs.
\subsection{Existing Work}
\label{sec:PreWrk}
Recently, computational NN approximations of PDEs that are based 
on \emph{variational formulations} of PDEs (unlike e.g. the PiNN 
formulations, which require additional assumptions for error bounds
\cite{zeinhofer2024unifiedframeworkerroranalysis})
have received considerable attention. 
Besides the so-called ``deep Ritz'' approaches \cite{EYuDeepRitz}
which minimize
variational energy functionals over NN-based, admissible 
parametric approximations, further lines of research which are
directly motivated by variational discretizations are, 
e.g., \cite{TJRVarMio24,AinsDong,ainsworth2024extended}.

\emph{Least Squares formulations of elliptic and parabolic PDEs}
which are variationally consistent, i.e., which admit unique
variational solutions that are identical to finite energy 
solutions resulting from, e.g., minimizing variational principles
over suitable sets of admissible functions,
have likewise received
considerable attention from a numerical analysis perspective,
in particular least squares formulations with $L^2$-based residuals 
\cite{BoGuLSQ}.
We mention the recent \cite{GaStLSQ} and references there for space-time 
LSQ approximation of parabolic PDEs, and \cite{FK23} for space-time
LSQ for the acoustic wave equation in a space-time domain.
Recently, a rather complete error analysis of LSQ FEM for 
linear, elliptic PDEs in bounded domains $\domain$ 
has been given in \cite{BB2023}.
To avoid undue regularity for the solution, in \cite{BB2023}
first order LSQ formulation of second order PDEs have been advocated.
This requires 
re-writing as a first order system (whence the name ``FoSLs'' for
the LSQ formulations). This involves, generally, the
use of so-called De Rham compatible spaces which we review below.

We also mention relevance 
of \emph{computable NN residuals} 
in numerical NN quality assessment and 
in loss-function design for practical NN training.
The present LSQ-based approach
uses $L^2(\domain)$-residuals of NNs in LSQ formulations of the PDE under consideration.
Approaches of this kind are referred also to as ``variational neural regression''
\cite{bachmayr2024variationallycorrectneuralresidual}.
These are not limited to PDEs, but also allow for
neural approaches to boundary integral equations (BIEs). 
However,
nonlocal, non-additive fractional Sobolev 
norms of variational BIE residuals forces 
\emph{loss-function design based on local, computable, a-posteriori error estimators}, 
see, eg., \cite{AHS22_989}.

Adaptive NN-growth in numerical PDE solution approximation has recently 
been proposed e.g. in \cite{dang2024adaptivegrowingrandomizedneural}, 
and the references there. While numerical evidence is provided for 
improved solution approximation, no theory of optimality of the proposed
methodologies is indicated.
\subsubsection{De Rham Compatible Spaces}
\label{sec:introfem}
On contractible $\domain \subset \R^d$, $d\geq 2$, 
the following sequence is exact
(e.g.\ \cite[Proposition 16.14]{ErnGuermondBookI2021}):
\begin{equation}\label{derham}
  \xymatrix{
    \R\ar^-{i}[r] &
    H^1(\domain)\ar^-{\operatorname{grad}}[r] &
    H(\curl,\domain) \ar^-{\curl}[r] &
    H(\divv,\domain)  \ar^-\divv[r] &
    L^2(\domain) \ar^-o[r] & \{0\}.
  }
\end{equation}
Here, the tag $i$ denotes `injection' and the tag $o$ denotes the zero operator.
This exact sequence is an example of a (more generally defined)
\emph{de Rham complex}.
Finite dimensional subspaces preserving this structure are usually
required to fit into a \emph{discrete de Rham complex}
(e.g.\ \cite[Proposition 16.15]{ErnGuermondBookI2021})
\begin{equation}\label{discderham}
  \xymatrix{
    \R\ar^-{i}[r] &
    \So(\mathcal{T},\domain) \ar^-{\operatorname{grad}}[r] &
    \Ne(\mathcal{T},\domain) \ar^-{\curl}[r] &
    \RT(\mathcal{T},\domain) \ar^-\divv[r] &
    \Sz(\mathcal{T},\domain) \ar^-o[r] &
    \{0\}.
  }
\end{equation}
Here, for a given, regular, simplicial partition
$\mathcal{T}$ of a contractible, polytopal domain $\domain$, 
$\So(\mathcal{T},\domain)$ stands for the class of
continuous piecewise linear functions 
($\Sop{1}$ in the notation of \cite{LODSZ22_991}),
$\Sz(\mathcal{T},\domain)$
stands for the class of piecewise constant functions
($\Szp{0}$ in the notation of \cite{LODSZ22_991}),
and
$\RT(\mathcal{T},\domain)$, $\Ne(\mathcal{T},\domain)$ 
denote lowest order
Raviart-Thomas and \Ned spaces, respectively 
(see, e.g., \cite{ErnGuermondBookI2021} or \cite{LODSZ22_991} 
for precise definitions).
The Finite Element spaces (FE for short) 
from \eqref{discderham} have the advantage of being
\emph{conforming}, i.e.,  each finite dimensional space 
is contained in the respective Sobolev space in \eqref{derham}.
Furthermore, the ($\calT$-dependent) projections
$\Pi_{\So},\Pi_{\Ne},\Pi_{\RT},\Pi_{\Sz}$ on these subspaces
introduced in \cite[Sec.\ 19.3]{ErnGuermondBookI2021} 
commute with the differential operators 
as shown in the following diagram \cite[Lemma 19.6]{ErnGuermondBookI2021}:
$$
\xymatrix{
  H^1(\domain)\ar^-{\operatorname{grad}}[r] \ar^{\Pi_{\So}}[d] &
  H^0(\curl,\domain) \ar^-{\curl}[r] \ar^{\Pi_{\Ne}}[d] &
  H^0(\divv,\domain) \ar^-\divv[r] \ar^{\Pi_{\RT}}[d] & L^2(\domain)
  \ar^{\Pi_{\Sz}}[d]
  \\
  \So(\mathcal{T},\domain) \ar^-{\operatorname{grad}}[r] &
  \Ne(\mathcal{T},\domain) \ar^-{\curl}[r] & \RT(\mathcal{T},\domain)
  \ar^-\divv[r] & \Sz(\mathcal{T},\domain)
}
$$
For these reasons we say that the spaces in \eqref{discderham} are \emph{de Rham compatible}.
While DeRham compatibility of approximation architectures and hypothesis classes
can be considered a convenient structural feature in the FoSLs reformulation
of scalar, elliptic and parabolic PDE in polytopal domains $\domain$, 
it is known to be essential in achieving variationally correct, convergent
approximations of solutions of field equations with constraints, such as the time-harmonic Maxwell equations
in electromagnetics \cite{CostabelCoerCMaxw}.
The corresponding FE spaces on triangulations $\cT$
of $\domain$ have been emulated by NN feature spaces in \cite{LODSZ22_991}.
\subsubsection{NN Emulation}
\label{sec:intronnfem}
The connection between NNs with 
Rectified Linear Unit (ReLU for short) activation 
and continuous piecewise linear (CPwL for short)
spline approximation spaces has been known for some time: 
\emph{nodal discretizations} based on CPwL Finite Element Methods (FEM) 
can be emulated by ReLU NNs 
(e.g.\ as introduced in \cite{ABMM2016} and \cite{HLXZ2020}):
the NNs' feature spaces comprise (NN emulations of) FE basis functions.

In \cite[Section~5]{LODSZ22_991}, 
we developed NN feature spaces of FE spaces 
\emph{exactly} for the DeRham compatible FE spaces 
$\So(\cT,\domain)$, $\Ne(\cT ,\domain)$, $\RT(\cT ,\domain)$ and $\Sz(\mathcal{T},\domain)$
on arbitrary regular, simplicial partitions $\cT$ of 
polytopal domains $\domain\subset \R^d$, $d\in\{2,3\}$.
This is the basis of the deep LSQ approach.

The NN basis emulations constructed in \cite{LODSZ22_991} 
are 
based on a combination of ReLU and BiSU (Binary Step Unit) activations.
We underline that our construction of NNs which emulate, in
particular, the classical ``Courant Finite Elements''
$\So(\mathcal{T},\domain)$, as well as $\Sz(\mathcal{T},\domain)$
and $\RT(\mathcal{T},\domain)$, applies on polytopal domains $\domain$
of any dimension $d\geq 2$.  
For the practically relevant space
$\So(\mathcal{T},\domain)$, in \cite{LODSZ22_991} we provided
ReLU NN constructions in arbitrary, finite dimension
$d\geq 2$ (the univariate case $d=1$ being trivial).
Apart from regularity of the
simplicial partition $\calT$ of $\domain$, no further constraints of
geometric nature are imposed on $\calT$.
As it is well-known in numerical analysis, however, 
convergence rate estimates or reliability and efficiency in adaptive
refinement procedures (which, as we shall show, will translate into
adaptive NN growth strategies) will require shape regularity of the 
partitions $\calT$ (see, e.g., \cite{EG2004} and the references there).

These constructions naturally extend to high order FE spaces.
See \cite[Section 7.1]{LODSZ22_991} for the exact NN emulation
of high order finite elements using ReLU, ReLU$^2$ and BiSU activation.\footnote{
ReLU$^2$ is defined as $x\mapsto \max\{0,x\}^2$ 
and the \emph{Binary Step Unit} (BiSU for short) 
is defined as $x\mapsto1$ for $x>0$ and $x\mapsto 0$ for $x\leq 0$.
For the exact emulation of high order continuous, piecewise polynomial functions,
it suffices to use ReLU and ReLU$^2$ activations, see \cite{HX2023,OS2024}.
For simplicity, we restrict our discussion in Section \ref{sec:NeuLSQ}
to lowest order spaces.
FoSLS NNs based on high order finite elements can be constructed analogously.}
\subsubsection{First Order System Least Squares for NN Approximation}
\label{sec:introfoslsnn}

Early results exploring a first order least squares approach
for neural network approximation include
\cite{CCLL2020,LZCC2022},
where in numerical experiments a discrete loss function was minimized. 
This loss function was obtained from the continuous loss function 
by quadrature and Monte Carlo sampling, 
respectively.
Recently, in \cite[Section 4]{BB2023} it was shown that for $N \in \N$,
considering a discretization of the integrals in the loss function
by Monte Carlo sampling in $N$ points,
it suffices to approximately minimize the discrete loss function
up to a tolerance that decreases to $0$ for $N\to\infty$.
It is then shown that under these conditions, 
the approximate solutions of the discrete problems
converge to the true solution of the PDE.

%
\subsection{Contributions}
\label{sec:Contr}
%
We propose FoSLS NNs as NN architectures for the numerical
approximation of linear elliptic, parabolic and hyperbolic PDEs.
They are based on the variational least squares (LSQ) formulation of the
PDE of interest, which is \emph{variationally correct} under minimal
regularity of the physical fields.
I.e., it is quadratic, strictly convex in the unknown fields, and
admits unique solutions which are known to coincide with the 
physically correct solution of the PDE of interest.
To this end, we propose a design of
FoSLS NNs which are \emph{structure-preserving},
i.e., 
their realizations as functions on a contractible spatial domain $\domain$
or on a spatiotemporal domain 
$\domain = (0,T) \times \spatialdomain$ 
must be de Rham compatible.
These properties result, usually, 
from an \emph{equivalent reformulation of variational principles 
in terms of LSQ functionals for first order systems of PDEs.}

Adopting De Rham compatible Finite Element spaces 
such as those in \eqref{discderham}, \cite{LODSZ22_991}
as feature spaces in FoSLS NNs
results in the \emph{transfer principle}, which allows to leverage 
known mathematical results on 
LSQ FE methods
for the analysis of FoSLS NN approximations.
In particular,
\begin{itemize} 
\item 
\emph{existing LSQ-FE approximation rate bounds 
      transfer} 
to expression rate bounds for the proposed FoSLS-based NNs.
\item 
\emph{homogeneous essential boundary conditions}
can be realized exactly, 
which is a severe issue in other approaches,
see \cite{CCLL2020,LZCC2022,BB2023} 
and the references there.
\item
\emph{physically correct, computable loss functions}: 
as is well-known and as we show, 
physically correct and numerically accessible 
LSQ functionals are available 
for a wide range of PDE boundary- and initial-boundary value problems.
I.e., 
there exist unique minimizers in function spaces of physical relevance 
of the LSQ functionals, 
resulting in \emph{computable numerical loss functions 
which are based on $L^2(\domain)$-norms of the FoSLS NN residuals}.
\item 
\emph{deterministic, high-order numerical quadratures}:
being essentially emulations of piecewise polynomial functions on
regular partitions of the physical domain $\domain$, 
the $L^2(\domain)$-based LSQ residual is numerically accessible via
(high-order) standard numerical integration.
\item 
\emph{computable expression errors in physically relevant norms}:
computable $L^2(\domain)$-based loss functions based on FoSLS formulations
are known to be equivalent to 
NN expression errors in physically relevant norms,
thereby allowing reliable and efficient numerical control of NN approximation errors
(subject to the assumption of polynomial source terms and boundary data, to
 discard data-oscillation error terms). 
\item 
\emph{localized loss functions}: 
additivity of the 
Lebesgue integral over triangulations $\cT$ of $\domain$, 
combined with the locality of differential operators 
implies 
\emph{loss functions composed additively of contributions from localized subnetworks},
corresponding to subdomains of the physical domain $\domain$.
\item
\emph{convergence}: standard (bisection) refinements of the partitions $\cT$
produces sequences of NNs of increasing width 
which are dense in the function spaces underlying
the LSQ form of the PDEs of interest.
\item 
\emph{provably rate-optimal, adaptive NN growth strategies}:
the use of loss functions based on the LSQ functional 
as computable error estimator, with local contributions 
driving adaptive mesh-refining strategies is well established 
 \cite{CCPark2015,CS2018,CollMarkLSQOptRat,PBringmann23}
in the adaptive LSQ finite element method (A-LSQFEM for short).
Combined with the FoSLS NN emulations developed in Section~\ref{sec:NeuLSQ},
this is used in Section~\ref{sec:AFEM} to infer 
\emph{practical algorithms for adaptive NN growth which produce 
sequences of convergent NN approximations of minimizers of LSQ functionals
\cite{FuePraeLSQ}.}
It follows that ALSQFEM-derived, adaptive NN growth strategies 
can result in rate-optimally convergent sequences of NNs, 
with a posteriori (i.e. upon completion of NN training)
numerically verifiable guarantees on error reduction. 
\end{itemize}
\subsection{Further Comparison with  ``Physics-Informed'' DL Approaches}
\label{sec:pinn}
%
We discuss in more detail 
some benefits of the considered
FoSLS-based approach for learning PDE solutions as compared to other
formulations, notably the so-called `PiNN' methodology.
See, e.g., \cite{DM2024} and references there.

Each FoSLS NN is of feedforward-type and 
realizes a de Rham-compatible finite element function
which is specified by the weights and biases of the NN.
We may consider some of these fixed, 
optimizing only the remaining weights and biases.
As usual, these are collected 
in the vector $\theta$ of trainable NN parameters.
In the FoSLS NNs based on \cite{LODSZ22_991},
which exactly emulate de Rham compatible finite elements 
from Section \ref{sec:introfem},
weights and biases in the hidden layers encode 
the mesh connectivity and the node positions.
The finite element shape functions are a basis of the FoSLS NN feature space,
and the NN weights in the output layer 
correspond to the finite element degrees of freedom.
By $\Theta$ we denote the set of all admissible values of $\theta$,
which give FoSLS NNs on regular, 
shape regular,
simplicial triangulations of the polytope $\domain$.
By $U_\theta$ we will denote 
the function realized by the FoSLS NN with weights and biases $\theta\in\Theta$.

We partition 
$\theta = (\theta_{\mathrm{hid}},\theta_{\mathrm{out}}) 
	\in \Theta_{\mathrm{hid}} \times \Theta_{\mathrm{out}} = \Theta$, 
where
$\theta_{\mathrm{out}}$ denotes the trainable weights in the output layer, and
$\theta_{\mathrm{hid}}$ the remaining trainable weights associated with the hidden layers
(recall that we consider only feedforward NNs without skip connections).
The standard FoSLS finite element methodology corresponds to 
only optimizing the output layer for given, 
fixed choices of admissible $\theta_{\mathrm{hid}}\in \Theta_{\mathrm{hid}}$. 
Then, $\Theta_{\mathrm{out}} = \R^N$ with 
$N$ denoting the dimension of the feature space, i.e., 
the number of weights in the output layer.
The linear dependence of the FoSLS residual on $\theta_{\mathrm{out}}$
means that
\emph{at fixed feature space parameters $\theta_{\mathrm{hid}}$,
an optimizer $\ths_{\mathrm{out}} \in \Theta_{\mathrm{out}}$ 
can be found by numerical solution of a linear LSQ problem},
and \emph{nonlinear optimization is not necessary} to determine
$\theta_{\mathrm{out}}$ at given, fixed $\theta_{\mathrm{hid}}$.

Let $\domain\subset\R^d$ for $d\in\N$, $d\geq2$
be a bounded, contractible, polytopal Lipschitz domain.
We explain the FoSLS NN methodology
with the (textbook) example of the Poisson equation 
(in the main body of this paper, we consider a much more general abstract setting,
see Section \ref{sec:AbsSet}, with 
Sections \ref{sec:Source}--\ref{sec:OCP} 
containing numerous concrete examples which are covered by this setting).

A first-order system formulation of the model Poisson equation
$-\Delta u = f$ in $\domain$ reads:
For $f\in L^2(\domain)$
find $u\in H^1_0(\domain)$ and $\bsigma \in H({\rm div};\domain)$
such that
\begin{equation}\label{eq:Poiss}
f+{\rm div}\bsigma = 0, \;\; \nabla u - \bsigma = \bszero \;\;\mbox{in}\;\domain, 
\;\; u= 0 \;\mbox{on}\;\partial\domain \;.
\end{equation}
In terms of the Hilbertian Sobolev space
$\bbV(\domain) := H^1_0(\domain) \times H({\rm div};\domain) $,
it can be written as:
find $U := ( u, \bsigma ) \in\bbV(\domain)$ 
such that
\begin{align*}
\VL U 
	:= \VL \left(\begin{array}{c} u \\ \bsigma \end{array}\right)
	:= \left(\begin{array}{c} -{\rm div} \bsigma \\ \nabla u - \bsigma \end{array}\right)
   = \left(\begin{array}{c} f \\ \bsnul \end{array} \right) 
   =: F \in L^2(\domain)^{d+1} =: \bbL(\domain) 
\;.
\end{align*}
The unique solution $U \in \bbV(\domain)$ of \eqref{eq:Poiss} 
satisfies
\[
U 
= 
{\arg}\min_{ V \in \bbV(\domain)} \| F - \VL V \|^2_{\bbL(\domain)}
= {\arg}\min_{(v,\btau) \in \bbV(\domain)} 
\left\{ 
\| f + {\rm div}\btau \|_{L^2(\domain)}^2
+
\| \nabla v - \btau \|_{L^2(\domain)^d}^2 
\right\} 
\;.
\]
We use \emph{conforming} FoSLS NNs
satisfying $U_\theta\in\bbV(\domain)$ 
for all admissible values $\theta\in\Theta$ of NN parameters $\theta$,
i.e. $\{ U_\theta : \theta \in \Theta\} \subset \bbV(\domain)$.
The error of the FoSLS NN approximation $U_\theta$ in the $\bbV(\domain)$-norm
is
\begin{equation}\label{eq:Err}
0\leq \calE(\theta) 
:= \norm[\bbV(\domain)]{ U - U_\theta } = \norm[\bbL(\domain)]{ F - \VL U_\theta },
\quad \theta \in \Theta \; .
\end{equation}
By \eqref{eq:Err}, the error vanishes if and only if $U_\theta = U$.
\emph{
For any given $U_\theta \ne U$, $\theta \in \Theta$,
\begin{equation}\label{eq:ErrFOSLS}
\calE_{FOSLS}(\theta) := \norm[\bbL(\domain)]{ F - \VL U_\theta }
\end{equation}
is, up to some numerical quadrature, 
a computable expression for the exact error $\calE(\theta)$ 
}
in terms of accessible data $F$, $\VL$ and $U_\theta$,
\emph{in the physically relevant norm} $\| \circ \|_{\bbV(\domain)}$.

For solving the PDE \eqref{eq:Poiss}
using a ``physics-informed'' approach, the least-squares residual
$\calE_{FOSLS}(\theta)$ in \eqref{eq:ErrFOSLS} 
serves as a \emph{numerically accessible loss function}, 
to be minimized computationally.

In all the least-squares formulations herein,
the \emph{computable FoSLS loss function} (up to data oscillation errors) 
is given by
$\calE_{\mathrm{FoSLS}}(\theta) = \norm[\bbL(\domain)]{ F - \VL U_\theta }$
where $\bbL(\domain)$ is generally $L^2(\domain)$, and accessible by
numerical quadrature.  
The FOSLS-loss additionally satisfies 
\begin{align}
\label{eq:introerrorproportional}
\calE_{\mathrm{FOSLS}}(\theta) 
= 
\norm[\bbL(\domain)]{ F - \VL U_\theta } 
\simeq 
\norm[\bbV(\domain)]{ U - U_\theta }
= 
\calE(\theta),
\end{align}
with $\simeq$ denoting either equality (as in the model problem \eqref{eq:Poiss})
or equivalence with absolute constants.

In particular, 
\emph{$\calE_{\mathrm{FoSLS}}(\theta)$ 
vanishes precisely when the exact, weak solution $U \in \bbV(\domain)$
of \eqref{eq:Poiss} is attained}. 
Minimizing 
$\calE_{\mathrm{FoSLS}}(\theta)$ over admissible NN approximations $U_\theta \in \bbV(\domain)$ 
provides a $\norm[\bbV(\domain)]{ \cdot }$-(quasi)optimal approximation 
$U_{\theta^*}$ of $U\in \bbV(\domain)$:
for all $\theta \in \Theta$ holds
\[
\| U - U_{\theta^*} \|_{\bbV(\domain)} 
\leq \| U - U_{\theta} \|_{\bbV(\domain)}   
=    \| F - \VL U_{\theta} \|_{\bbL(\domain)}
= \calE_{\mathrm{FoSLS}}(\theta)
\;.
\]
The \emph{computable numerical residual} $\calE_{\mathrm{FOSLS}}(\theta) = \norm[\bbL(\domain)]{ F - \VL U_\theta }$, 
i.e. 
\emph{the numerical value of the loss function during training}
is a computable  upper bound for the solution approximation error 
$\norm[\bbV(\domain)]{ U - U_\theta }$
in the physically meaningful ``energy'' norm $\norm[\bbV(\domain)]{ \cdot }$
\emph{of any NN approximation} $U_\theta$. 
This is relevant for 
numerical approximations resulting from incomplete training where, as a rule,
$\theta \ne \ths$.

Now, we consider a discretization of $\calE_{\mathrm{FoSLS}}$
and an approximate minimizer of this discretized loss.
As usual, the incurred error can be decomposed into 
an approximation error, generalization gap and an optimization error.
We define the generalization error in terms of the FoSLS loss, i.e.
$\calE_G(\cdot) := \calE_{\mathrm{FoSLS}}(\cdot)$.
Given a set of interpolation points $S \subset \overline\domain$,
we denote by $\calE_T(\cdot,S)$
a discretization of $\calE_G(\cdot)$
which only depends on point values in $S$.
Let $\ths\in\Theta$ be such that $\calE_T(\ths,S)$ is small.
We recall the error decomposition from e.g. \cite[Equation (3.11)]{DM2024}.
For all $\thh\in\Theta$:
\begin{align*}
\calE_G(\ths) 
	\leq \calE_G(\thh) 
		+ 2 \sup_{\theta\in\Theta} \snorm{ \calE_G(\theta) - \calE_T(\theta,S) }
		+ \snorm{ \calE_T(\ths,S) - \calE_T(\thh,S) }
.		
\end{align*}
The first term on the right-hand side is an upper bound for 
$\inf_{\theta\in\Theta} \calE_G(\theta)$, which is the \emph{approximation error}.
For the NNs in the present note, it can be bound with standard Finite-Element based
error bounds.
The second term is called \emph{generalization gap}
and measures to which extent smallness of the discrete loss implies 
smallness of the continuous loss.
The third term is the \emph{approximation error},
i.e. the quality of the algorithm by which 
the approximate minimizer $\ths$ of $\calE_T(\cdot,S)$
has been determined.

In Section \ref{sec:deepFoSLS},
we show that FoSLS NNs exactly emulate finite element spaces
appearing in the discrete exact sequence \eqref{discderham}.
Therefore, the generalization error can be estimated using existing finite element theory.

The functions emulated by the FoSLS NNs introduced in Section \ref{sec:deepFoSLS}
exactly realize piecewise polynomial functions 
on a regular, simplicial partition $\cT$ of the physical domain $\domain$.
Note that these piecewise polynomial functions need not be continuous.
Considering the case that
$F$ is also a piecewise polynomial function on $\cT$, 
on each element $K\in \cT$, 
the residual $\VL u_\theta - F$ restricted to $K$ 
is a polynomial of finite, known degree.
We can define quadrature points that integrate it exactly.
Collecting these quadrature points in a set $S$,
we obtain $\calE_T(\theta,S) = \calE_G(\theta)$ for all $\theta\in\Theta$.
I.e., for \emph{collocation points based on mesh-adapted quadratures, 
there is no generalization gap} provided $F$ is a piecewise polynomial function
on $\cT$.
%

Concerning the approximation error,
if $\theta_{\mathrm{hid}} \in \Theta_{\mathrm{hid}}$ is fixed,
then the squared loss function $\calE_T(\theta,S)^2$ 
depends quadratically on 
$\theta_{\mathrm{out}}(\theta_{\mathrm{hid}}) \in \Theta_{\mathrm{out}} = \R^N$.

The optimal $\ths_{\mathrm{out}}(\theta_{\mathrm{hid}})$ 
can be computed via a linear LSQ problem with
standard numerical linear algebra such as PCCG or thin QR decompositions.
The LSQ solution realizes a $\bbV(\domain)$ best approximation of the solution $U$ 
in the NN feature space corresponding to $\theta_{\mathrm{hid}}$.
\subsection{Layout}
\label{sec:outline}
This text is structured as follows.  
In Section~\ref{sec:LSQForm}, 
we recapitulate FoSLS (re)formulations of a wide range of
source and optimal control problems for 
linear elliptic and parabolic PDEs in a bounded, contractible,
polytopal physical domain $\domain\subset \bbR^d$.
Some axiomatic setting is presented, 
inspired by recent work \cite{FuePraeLSQ}
on adaptive LSQ FEM for elliptic PDEs.
Also, source problems for
linear, parabolic evolution equations in space-time cylinders 
$\domain = (0,T) \times \spatialdomain$ 
are considered, following \cite{FuePraeLSQ,GaStLSQ,GaStLSQApplic},
and for the
time-domain acoustic wave equation, based on \cite{FK23}.
Furthermore,
we cover abstract optimal control problems, following \cite{TFMKOptCtrl2022}.

Section \ref{sec:LSQFEM} recapitulates facts on
the least squares Galerkin method
in abstract form and the quasi-optimality of its solution.

Section~\ref{sec:NeuLSQ} recapitulates notation from our recent work
\cite{LODSZ22_991} on 
NNs which are de Rham compatible on regular, simplicial
partitions of the polytope $\domain$.
Several NNs from \cite{LODSZ22_991} are combined
to emulate FE spaces of all the variables of the first order 
linear system formulations of the (initial) boundary value problems
introduced in Section~\ref{sec:LSQForm}.
The NNs minimizing the corresponding LSQ loss are shown to be quasi-optimal.

In Section \ref{sec:AFEM}, we describe an adaptive finite element algorithm
which converges for the PDEs presented in Section \ref{sec:LSQForm}.
When combined with the NN emulations from Section \ref{sec:NeuLSQ},
this algorithm provides an \emph{adaptive neural network growth strategy},
which is driven by \emph{localized} (in physical domain and in the NN) 
contributions to the global LSQ loss function.
These contributions quantify localized ``out-of-equilibrium'' of the current 
NN state with respect to the PDE LSQ residual.
Section \ref{sec:Concl} concludes the paper.
%
\section{Least Squares Formulation}
\label{sec:LSQForm}
In Section~\ref{sec:AbsSet}, we present an abstract setting of 
\emph{First order System Least Squares (FoSLS) formulations}
of \emph{well-posed PDEs}. 
FoSLS formulations are valid 
under low regularity assumptions on the domain $\domain$ which
we shall assume to be bounded, contractible, 
polytopal with Lipschitz boundary $\partial\domain$.
All FoSLS-formulations involve an  $L^2(\domain)$-norm 
of the residual of the PDE with respect to arguments 
in suitable function spaces,
forming, for second order elliptic systems on contractible $\domain$,
exact sequences.
In Section~\ref{sec:Source}, 
we illustrate this by (textbook) examples of LSQ formulations
of source problems for the Poisson equation and the Helmholtz equation,
for linear elastostatics in dimensions $d=2,3$ 
and time-harmonic electromagnetics (Maxwell's equations)
in dimension $d=3$,
space-time LSQ formulations of the
advection-reaction-diffusion equation and the acoustic wave equation.
Optimal control problems also admit LSQ formulations, 
as we will see in Section \ref{sec:OCP}.
Other examples covered by the present, unified approach
include (but are not limited to)
the Stokes problem, cf. \cite[Section 3.5]{FuePraeLSQ} 
and also \cite[Section 3.2]{TFMKOptCtrl2022}
for the corresponding optimal control problem.
\subsection{Abstract Setting}
\label{sec:AbsSet}
We assume throughout a physical
domain $\domain\subset \bbR^d$ for $d\in\N$, $d\geq 2$
which satisfies
\begin{itemize}
\item[(D)]
Domain assumptions: 
$\domain$ is a bounded, contractible, polytopal Lipschitz\footnote{
Cf. e.g. \cite[Definition 3.2]{ErnGuermondBookI2021}.}
domain.
\end{itemize}
For
$F\in \bbL(\domain) := L^2(\domain_1) \times \cdots \times L^2(\domain_n)$
for some $n\geq 1$ and polytopal domains $\domain_1,\ldots,\domain_n$ satisfying (D),
consider a linear PDE in the generic form
\be\label{eq:PDE}
\VL U = F \quad \mbox{in} \;\; \bbL(\domain) \;.
\ee
Here, 
$\VL \in \calL(\bbV(\domain), \bbL(\domain))$ 
is an operator defined on some Hilbertian Sobolev space 
$\bbV(\domain)$ 
(carrying homogeneous essential boundary conditions for \eqref{eq:PDE})
satisfying the following assumptions.
\begin{itemize}
\item[(A1)] 
{\bf Well-posedness}
$\VL \in \calL(\bbV(\domain), \bbL(\domain))$ 
is boundedly invertible:
there are constants $0 < c_{\VL} \leq C_{\VL} < \infty$ 
such that
$$
\forall v\in \bbV(\domain): \quad 
c_\VL \| v \|_{\bbV(\domain)} 
\leq 
\| \VL v \|_{{\bbL(\domain)}} 
\leq 
C_\VL \| v \|_{\bbV(\domain)} 
\;.
$$
%

The inner product corresponding to $\norm[\bbL(\domain)]{\cdot}$ 
will be denoted by $(\cdot,\cdot)_{\bbL(\domain)}$.
Thus, the least squares bilinear form 
$\bbV\times\bbV\to\R: (v,w) \mapsto ( \VL v, \VL w )_{\bbL(\domain)} $ 
is assumed coercive with coercivity constant $c_\VL^2$.
\item[(A2)] 
{\bf Existence} of solutions: 
the data satisfies $F\in {\rm range}(\VL) \subset \bbL(\domain)$.
\end{itemize}
From (A2), for given $F$, there exists a solution $U\in \bbV(\domain)$ of \eqref{eq:PDE}.
With (A1), we find the \emph{error-residual relation}: 
for every $v\in \bbV(\domain)$, there holds
\be\label{eq:PDEUniq}
c_\VL \| U - v \|_{\bbV(\domain)}
\leq
\| F - \VL v \|_{\bbL(\domain)}
\leq 
C_\VL \| U - v \|_{\bbV(\domain)}
\;.
\ee
This immediately implies uniqueness of the solution $U \in \bbV(\domain)$. 
This solution is, in particular, the unique minimizer of the 
\emph{least squares functional} $\LS( \cdot; F )$ 
associated to \eqref{eq:PDE}:
\begin{equation}\label{eq:FoSLSMin}
U = {\rm arg}\min_{v\in \bbV(\domain)} 
\LS(v;F), 
\quad\mbox{with}\quad
\LS(v;F) := \| F - \VL v \|_{\bbL(\domain)}^2 
\;.
\end{equation}
We spell out some further structural hypotheses for $\bbV(\domain)$, $\VL$ 
and the norm $\| \circ \|_{\bbV(\domain)}$.
For many PDEs of interest, $\VL$ is an integer order differential operator,
i.e. in particular a local operator, and $\bbV(\domain)$ is an integer order
Sobolev space in $\domain$. 
The corresponding (Hilbertian) norm $\| \circ \|_{\bbV(\domain)}$ 
then satisfies
\begin{itemize}
\item[(A3)] {\bf $\| \circ \|_\domain$ Norm additivity}: 
for any disjoint, measurable subsets 
$\omega_1,\omega_2\subset \domain$
it holds that 
$$
\forall v\in \bbV(\domain): \quad 
\| v \|_{\bbV(\omega_1\cup \omega_2)}^2 
= 
\| v \|^2_{\bbV(\omega_1)} + \| v \|^2_{\bbV(\omega_2)}\;
\quad 
\| \VL v \|_{\bbL(\omega_1\cup \omega_2)}^2 
=
\| \VL v \|_{\bbL(\omega_1)}^2 + \| \VL v \|_{\bbL(\omega_2)}^2 
\;.
$$
\item[(A4)]
{\bf Norm continuity}: 
for measurable subdomains $\omega\subset \domain$ 
holds
$$
\forall v\in \bbV(\domain):\quad 
\| v \|_{\bbV(\omega)} \to 0 \;\; \mbox{as} \;\; | \omega | \to 0 \;.
$$
Here, $|\omega|$ denotes the Lebesgue measure of the domain $\omega \subset \domain$.
\end{itemize}

In the LSQ formulation \eqref{eq:FoSLSMin},
for given data $F$, 
the quadratic (for a linear $\VL$) functional $v\mapsto \LS(v;F)$
defined in \eqref{eq:FoSLSMin} takes the role of loss-functional in NN training.
Due to the additivity (A3) and the locality of $\VL$, the numerical
evaluation of $\LS(v;F)$ amounts to a sum of local integrals
over a regular, simplicial triangulation $\cT$ of $\domain$.

\subsection{Source Problems}
\label{sec:Source}
We now list several concrete LSQ formulations for linear PDEs.
We address in particular so-called source problems, where the
interest is in the LSQ characterization of the weak solution
for a given instance of input data. We consider both, 
stationary, elliptic PDEs (Sections~\ref{sec:Poisson} -- \ref{sec:nLSQCEM})
parabolic PDEs (Section~\ref{sec:nLSQParab}) and
hyperbolic PDEs (Section~\ref{sec:nLSQWave}).
\subsubsection{Poisson Equation and Helmholtz Equation}
\label{sec:Poisson}
In \eqref{eq:PDE},
second order, elliptic differential operators $\VL$ would require $\bbV(\domain)$
to be a subspace of the Sobolev space $H^2(\domain)$. 
We illustrate this for the Dirichlet problem
of the Poisson equation (corresponding to $k=0$ in what follows) 
and the Helmholtz equation at wavenumber $k>0$.

With the choice $\bbV(\domain) = (H^2\cap H^1_0)(\domain)$,
it is, however, well-known that for $\domain\subset \bbR^2$ being 
a nonconvex polygon,
$-\Delta : \bbV(\domain) \to L^2(\domain)$ is not an isomorphism, due to
non-$H^2$ singularities appearing in the solution 
$u\in H^1_0(\domain)$ 
of the Poisson equation $-\Delta u = f$
for $f\in L^2(\domain)$ at reentrant corners (e.g. \cite{Grisvard}).
See the recent \cite{ainsworth2024extended} for a 
\emph{residual formulation of a second-order elliptic PDE
in weighted $L^2(\domain)$-norms in polygons addressing this issue;
the FoSLS approach avoids the use of weighted norms, at the expense of
introducing additional fields to be approximated.}

Condition (A2) is restored by choosing $n>1$ 
and 
by writing $\VL = -\Delta-k^2$ as a first-order system: 
with $\bsigma = \nabla u$, one arrives at
$$
f+{\rm div}\bsigma +k^2 u = 0, \;\; \nabla u - \bsigma = \bszero \;\;\mbox{in}\;\domain, 
\;\; u= 0 \;\mbox{on}\;\partial\domain \;.
$$
This can be done in several ways. 
For example,
with the unknown $U = (u, \bsigma)$, $n=d+1$, 
and the choice of spaces
$\bbV(\domain) := H^1_0(\domain) \times H({\rm div};\domain)$
from the de Rham complex \eqref{derham} (with homogeneous boundary conditions in the first component).
As norm on this space we choose 
$\| (v,\bstau) \|_{\bbV(\domain)}^2
:=
\| v \|_{L^2(\domain)}^2 
+
\| \nabla v \|_{L^2(\domain)^d}^2 
+
\| \bstau \|_{L^2(\domain)^d}^2 
+
\| \divv \bstau \|_{L^2(\domain)}^2$.
One attains the form \eqref{eq:PDE} 
with the \emph{first order differential operator} 
\begin{align}
\label{eq:PoissonFoSLS}
\VL U 
	:= \VL \left(\begin{array}{c} u \\ \bsigma \end{array}\right)
	:= \left(\begin{array}{c} -{\rm div} \bsigma -k^2 u \\ \nabla u - \bsigma \end{array}\right)
   = \left(\begin{array}{c} f \\ \bsnul \end{array} \right) 
   =: F \in L^2(\domain)^{d+1} =: \bbL(\domain) 
\;.
\end{align}
Assumptions (A3)--(A4) hold for
$\bbV(\domain) = H^1_0(\domain) \times H({\rm div};\domain)$,
and (A1)--(A2) follow from the following proposition.
\begin{proposition}[{{\cite[Section 3.1]{FuePraeLSQ}}}]
\label{prop:Poisson}
Let $\domain \subset \R^d$ be a bounded, simply connected, polytopal domain with 
Lipschitz boundary $\partial\domain$. 
If either $k=0$ or if $k^2>0$ is not an eigenvalue of the Dirichlet-Laplacian in $\domain$, 
then $\VL$ is an isomorphism.
\end{proposition}

\subsubsection{Linear Elastostatics}
\label{sec:Elasticity}
Linear elasticity can be treated similarly to the Poisson equation in the previous section,
as we recall from \cite{CKS2005} and \cite[Section 3.3]{FuePraeLSQ}.
We consider $\domain\subset\R^d$ for $d\in\{2,3\}$.
Given $\bsf\in L^2(\domain)^d$ and positive Lam\'e constants $\lambda,\mu>0$, 
the linear elasticity system reads as follows:
Find a vector-valued displacement field 
$\bsu : \domain \to \R^d$, $\bsu\in H^1_0(\domain)^d$
and a matrix-valued stress field 
$\bsigma : \domain \to \R^{d\times d}$, $\bsigma\in H(\divv;\domain)^d$
such that 
\begin{align}
\label{eq:Elasticity}
- \divv \bsigma = \bsf,
\quad
\bsigma = \calC \bseps
\quad\text{ in }\quad
\domain
.
\end{align}
The linear strain tensor $\bseps : \domain \to \R^{d\times d}$ is defined as
$\bseps(\bsu) = \tfrac12 ( \nabla\bsu + (\nabla\bsu)^\perp )$
and the elasticity tensor $\calC$ is such that
$\calC \bseps = 2 \mu \bseps + \lambda (\trace \bseps) I_{d\times d}$,
cf. e.g. \cite[Section 3.3]{FuePraeLSQ}.
Here, 
$\nabla\bsu$ denotes the Jacobian,
$\divv \bsigma$ is the row-wise divergence,
$\trace \bseps$ is the trace
and $I_{d\times d}$ denotes the $d\times d$ identity matrix.
Equation \eqref{eq:Elasticity} 
can be written as a first order system in terms of the Hilbert space
\begin{align*}
\bbV(\domain) = H^1_0(\domain)^d \times H(\divv;\domain)^d
,
\end{align*}
equipped with the norm
\begin{align}
\label{eq:bbVElasticitynorm}
\| (\bsv,\bstau) \|_{\bbV(\domain)}^2
:=
\| \calC^{1/2} \bseps(\bsv) \|_{L^2(\domain)^{d\times d}}^2 
+
\| \calC^{-1/2} \bstau \|_{L^2(\domain)^{d\times d}}^2 
+
\| \divv \bstau \|_{L^2(\domain)^d}^2
\;.
\end{align}
As these norms are local, (A3) and (A4) hold.
Finally, we obtain an equation of the form \eqref{eq:PDE}:
\begin{equation}\label{eq:FoSLSElasticity}
\VL U
:=
\VL \left(\begin{array}{c} \bsu \\ \bsigma \end{array}\right) 
:=
\left(\begin{array}{c} - \divv \bsigma \\ 
	\calC^{-1/2} \bsigma - \calC^{1/2} \bseps(\bsu) \end{array}\right) 
=
\left(\begin{array}{c} \bsf \\ \bsnul \end{array}\right) 
=: F \in L^2(\domain)^{d+d^2}
=: \bbL(\domain)
\;.
\end{equation}
Assumptions (A1) and (A2) also hold, as stated in the following proposition.
\begin{proposition}[{{\cite[Theorem 2.1]{CKS2005}, \cite[Section 3.3]{FuePraeLSQ}}}]
\label{prop:Elasticity}
For a polyhedral domain $\domain\subset \R^d$, $d\in\{2,3\}$ satisfying (D),
$\VL$ is an isomorphism.

For all $\bsf \in L^2(\domain)^d$,
\eqref{eq:Elasticity} admits the FoSLS formulation
\eqref{eq:FoSLSElasticity}
satisfying (A1)--(A2).
Equation \eqref{eq:FoSLSElasticity} 
admits a unique solution $U = (\bsu,\bsigma)\in \bbV(\domain)$,
which is the unique minimizer over 
$\bbV(\domain)$
of the LSQ functional \eqref{eq:FoSLSMin} 
with $\VL$ and $F$ as in \eqref{eq:FoSLSElasticity}.
\end{proposition}
\subsubsection{Time-Harmonic Electromagnetic Waves (Maxwell Equations)} 
\label{sec:nLSQCEM}
We present a FoSLS formulation of the Maxwell equations,
following \cite[Section 3.4]{FuePraeLSQ}. 
We assume
$\domain\subset \bbR^3$ (we do not consider $d=2$ for this problem class).
For given $\bsf\in L^2(\domain)^3$ and $c\in L^\infty(\domain)$
such that either $\essinf_{x\in\domain} c(x) > 0$
or $c = - \omega^2$ is not an eigenvalue of the cavity problem (cf. \cite[Section 2.4]{CS2018}),
the Maxwell equations in $\domain$ read:
Find 
$\bsu\in H_0(\curl;\domain)$ 
and  
$\bsigma \in H(\curl;\domain)$
such that
\begin{equation}\label{eq:Maxwell}
\curl \bsigma + c\bsu = \bsf,  
\quad 
\curl \bsu - \bsigma = \bsnul \quad\mbox{in}\quad\domain,
\end{equation}
where 
$$
H(\curl;\domain) := \{ \bsv \in L^2(\domain)^3: \curl \bsv \in L^2(\domain)^3 \}
,\;\;
H_0(\curl;\domain) := \{ \bsv \in H(\curl;\domain) : \bsv\times \bsn |_{\partial\domain} = 0 \}.
$$
We choose the Hilbert 
space $\bbV(\domain)$ as
$$
\bbV(\domain) := H_0(\curl;\domain) \times H(\curl;\domain)
$$
equipped with the graph norm on $\bbV(\domain)$ 
given by \cite[Section 3.4]{FuePraeLSQ}
$$
\| (\bsv,\bstau) \|_{\bbV(\domain)}^2
:=
\| \bsv \|_\domain^2 
+
\| \curl \bsv \|_\domain^2 
+
\| \bstau \|_{\domain}^2 
+
\| \curl \bstau \|_{\domain}^2
\;,
$$
with $\| \circ \|_\domain$ denoting the $L^2(\domain)^3$-norm.
These norms being local, (A3) and (A4) obviously hold.

Thus, 
the Maxwell system \eqref{eq:Maxwell} 
fits into the generic format \eqref{eq:PDE}
with 
\begin{equation}\label{eq:FoSLSMaxwell}
\VL U
:=
\VL \left(\begin{array}{c} \bsu \\ \bsigma \end{array}\right) 
:=
\left(\begin{array}{c} \curl \bsigma + c \bsu \\ \curl \bsu - \bsigma \end{array}\right) 
=
\left(\begin{array}{c} \bsf \\ \bsnul \end{array}\right) 
=: F \in L^2(\domain)^6
=: \bbL(\domain)
\;.
\end{equation}
Assumptions (A1) and (A2) hold by the following proposition.
\begin{proposition}[{{\cite[Section 3.4]{FuePraeLSQ}, \cite[Section 3.3]{TFMKOptCtrl2022}}}]
\label{prop:Maxwell}
For a polyhedral domain $\domain\subset \R^3$ satisfying (D),
for all $c\in L^\infty(\domain)$
such that either $\essinf_{x\in\domain} c(x) > 0$
or $c = - \omega^2$ is not an eigenvalue of the cavity problem,
$\VL$ is an isomorphism.
For all $\bsf \in L^2(\domain)^3$,
\eqref{eq:Maxwell} admits the FoSLS formulation
\eqref{eq:FoSLSMaxwell}
satisfying (A1) and (A2).
The FoSLS formulation \eqref{eq:FoSLSMaxwell} 
admits a unique solution $U = (\bsu,\bsigma)\in \bbV(\domain)$ 
being the unique minimizer over 
$\bbV(\domain)$
of the LSQ functional \eqref{eq:FoSLSMin} 
with $\VL$ and $F$ as in \eqref{eq:FoSLSMaxwell}.
\end{proposition}
\subsubsection{Space-Time LSQ for Advection-Reaction-Diffusion}
\label{sec:nLSQParab}

In a bounded polytopal domain 
$\spatialdomain\subset\bbR^d$ 
satisfying condition (D),
with boundary 
$\Gamma = \partial \spatialdomain$
where homogeneous Dirichlet boundary conditions are imposed,
and 
a finite time interval $I = (0,T)$, 
we consider the \emph{parabolic initial boundary value problem}:
given data 
$f, \bsa, \bsb, c, u_0$, 
find
$u:I\times \spatialdomain \to \bbR$ 
such that
\begin{equation} \label{eq:Parabolic}
\begin{array}{rcl} 
\partial_t u - \divv_x(\bsa \nabla_x u) + \bsb\cdot\nabla_x u + cu & = & f 
\quad \mbox{in}\quad I\times \spatialdomain,
\\
u & = & 0 
\quad \mbox{on} \quad I\times \Gamma,
\\
u(0,\cdot) & = & u_0 
\quad \mbox{in} \quad \spatialdomain.
\end{array}
\end{equation}
%
Here, 
we have $\domain := I \times \spatialdomain$,
$\divv_x$ and $\nabla_x$ denote differential operators on $\Omega$
and the diffusion coefficient
$\bsa = \bsa^\top \in L^\infty(\domain)^{d\times d}$
is a given, symmetric matrix function which is assumed
to be uniformly positive definite, 
$0\leq c\in L^\infty(\domain)$ is a reaction coefficient,
and 
$\bsb\in L^\infty(\domain)^d$ an advection field 
such that for all $u,v\in H^1_0(\spatialdomain)$
\begin{subequations}
\label{eq:ParabSpatialwellp}
\begin{align}
\label{eq:ParabContin}
\snormc{ \int_\spatialdomain \bsa(t,\cdot) \nabla_x u \cdot \nabla_x v + v \bsb(t,\cdot) \cdot \nabla_x u + c (t,\cdot) u v } 
	\leq &\, C \norm[H^1(\spatialdomain)]{ u } \norm[H^1(\spatialdomain)]{ v },
	\\
\int_\spatialdomain \bsa (t,\cdot) \nabla_x v \cdot \nabla_x v 
                 + v \bsb(t,\cdot) \cdot \nabla_x v + c(t,\cdot) v^2  
	\geq &\, C' \norm[H^1(\spatialdomain)]{ v }^2 - C'' \norm[L^2(\spatialdomain)]{ v }^2,
\label{eq:ParabGarding}
\end{align}
\end{subequations}
for all $t\in I$, with some
constants $C,C'>0$, $C''\geq0$ independent of $t$.
In addition, we consider a right-hand side $f\in L^2(\domain)$
and an initial value $u_0\in L^2(\spatialdomain)$.

We introduce the Hilbert-space for the FoSLS space-time formulation:
with 
$$
U = (u_1,\bsu_2), \quad\mbox{where} 
\quad
u_1 = u, \;  
\bsu_2 = -\bsa \nabla_x u 
\;,
$$
we obtain from \eqref{eq:Parabolic} the first order system
\begin{equation}\label{eq:Para1stOrd}
\VL U
:=
\VL
\left(\begin{array}{c} 
u_1
\\
\bsu_2
\end{array}
\right)
:= 
\left(
\begin{array}{c} 
\bsu_2 +\bsa\nabla_x u_1
\\
\partial_t u_1 +\divv_x \bsu_2 + \bsb\cdot \nabla_x u_1 + cu_1
\\
u_1(0,\cdot) 
\end{array}
\right)
=
F
:= 
\left(\begin{array}{c} 
\bsnul
\\
f
\\
u_0
\end{array}
\right)
\end{equation}
where 
$F\in L^2(\domain)^d \times L^2(\domain) \times L^2(\spatialdomain)
	=: \bbL(\domain)$.
 
The (space-time) LSQ formulation of \eqref{eq:Para1stOrd} is based on
\begin{equation}\label{eq:heatV}
\bbV(\domain) := 
\{ \bsu = (u_1,\bsu_2) \in L^2(I;H^1_0(\spatialdomain))\times L^2(I\times \spatialdomain)^d: \divv \bsu \in L^2(I\times \spatialdomain) \}\;,
\end{equation}
where $\divv\bsu = \partial_t u_1 + \divv_x \bsu_2$ 
is the divergence on the space-time cylinder $\domain$.
We impose on $\bbV(\domain)$ 
the norm\footnote{\label{fn:tracetimezero}
In \cite[Equation (2.1)]{GaStLSQ}, the last term is omitted 
from the definition of $\norm[\bbV(\domain)]{\cdot}$.
That norm is equivalent to the one defined here,
which is the norm that is used in the proof of \cite[Theorem 3.3]{GaStLSQ}
and introduced in Step 1 of the proof of that theorem.
}
$\| \circ \|_{\bbV(\domain)}$ given by
\begin{align}
\label{eq:bbVheatnorm}
\| \bsu \|_{\bbV(\domain)}^2 
:= 
\| u_1 \|^2_{L^2(I;H^1(\spatialdomain))} 
	+ \| \bsu_2 \|^2_{L^2(I;L^2(\spatialdomain)^d)} 
	+ \| \divv \bsu \|^2_{L^2(\domain)}
	+ \norm[L^2(\spatialdomain)]{ u_1(0,\cdot) }^2 
\;.
\end{align}
This norm satisfies Assumptions (A3) and (A4), 
as only a local (differential) operator enters the LSQ residual.
%
%
%

Conditions (A1) and (A2) follow from the fact that $\VL$ is an isomorphism,
which is in this case the assertion of \cite[Theorem 2.3 and Remark 2.4]{GaStLSQ}.
We use the formulation of the advection term from \cite[Remark 2.4]{GaStLSQ}.
This formulation equals that of \cite[Theorem 2.2]{GaStLSQApplic} 
up to a factor $-1$ in the first component.
%
%
See also \cite{FuePraeLSQ} for the analysis of a least squares formulation 
for a particular case of \eqref{eq:Parabolic}.

With conditions (A1)--(A2) 
in hand 
we obtain the following result.
\begin{proposition}[{{\cite[Theorem 2.3 and Remark 2.4]{GaStLSQ}}}]
\label{prop:heat}
For all $\bsa$, $\bsb$ and $c$ 
such that \eqref{eq:ParabSpatialwellp} is satisfied
and
for all $f\in L^2(\domain)$, $u_0 \in L^2(\spatialdomain)$,
there exists a unique solution 
$U = (u_1,\bsu_2) \in \bbV(\domain)$
of \eqref{eq:Para1stOrd}.
It is, for given data $F$ as in \eqref{eq:Para1stOrd}, 
the unique minimizer of the LSQ functional  
$\LS(\cdot;F): \bbV(\domain) \to \bbR$
given for $\bsv = (v_1,\bsv_2)\in \bbV(\domain)$
by
\begin{align}
\LS(\bsv;F)
:= &\,
\| \bsv_2 +\bsa\nabla_x v_1\|^2_{L^2(\domain)^d}
+
\| \partial_t v_1 +\divv_x \bsv_2 +\bsb\cdot \nabla_x v_1 + cv_1 - f \|^2_{L^2(\domain)} 
\nonumber
\\
&\,
+ 
\|  v_1(0,\cdot)-u_0 \|^2_{L^2(\spatialdomain)} 
\;.
\label{eq:LSQLHeat}
\end{align}
\end{proposition}
A LSQ formulation
with nonhomogeneous pure Dirichlet or Neumann boundary conditions
is available in \cite[Section 2.2]{GaStLSQ}.
Also the pure advection case is considered recently in \cite{cai2023leastsquares}.
\subsubsection{Space-Time LSQ for the Acoustic Wave Equation}
\label{sec:nLSQWave}
Retaining notation from the preceding Section~\ref{sec:nLSQParab},
in $I\times \spatialdomain$, the propagation of acoustic waves 
in a homogeneous, isotropic medium (extensions of the following 
to anisotropic media are straightforward) 
is governed by the 
\emph{acoustic wave equation} in the 
space-time cylinder $\domain = I \times \spatialdomain$,
\begin{equation}\label{eq:AcWave2ndOrd}
\begin{array}{rcl}
\partial_{tt} u -\Delta_x u & = & f \;\;\mbox{in}\;\; I\times \spatialdomain, 
\\ 
u&=&0 \;\;\mbox{on}\;\; I\times \Gamma,
\\
u(0,\cdot) & = & u_0 \;\; \mbox{in} \;\; \spatialdomain,
\\ 
\partial_tu(0,\cdot) &=& u_1 \;\; \mbox{in} \;\; \spatialdomain
.
\end{array}
\end{equation}
For the first order formulation, 
\cite{FK23} proposes the choices $v:=\partial_t u$, 
$\bsigma := \nabla_x u$ and $\bg = \bsnul$, $v_0=u_1$, $\bsigma_0 = \nabla_x u_0$.
This results in the 
\emph{first order formulation of the linear, acoustic wave equation}
\begin{equation}\label{eq:AcWavefosls}
\begin{array}{rcl}
\partial_t v - \divv_x \bsigma &=& f \;\; \mbox{in} \;\; I\times\spatialdomain,
\\
\partial_t \bsigma - \nabla_x v &=& \bg \;\; \mbox{in} \;\; I\times\spatialdomain,
\\
v &=& 0 \;\; \mbox{on} \;\; I\times\Gamma,
\\
v(0,\cdot) &=& v_0 \;\; \mbox{in} \;\; \spatialdomain,
\\
\bsigma(0,\cdot) &=& \bsigma_0 \;\; \mbox{in} \;\; \spatialdomain 
\;.
\end{array}
\end{equation}
This system fits into the abstract setting \eqref{eq:PDE} 
if we introduce the 
\emph{first order acoustic wave operators} 
\begin{equation}\label{eq:AcWvOps}
\VL U
:=
\VL(v,\bsigma) 
:= 
\left(\begin{array}{c} \calA_0(v,\bsigma) \\ v(0,\cdot) \\ \bsigma(0,\cdot) \end{array}\right)
\;,
\quad 
\mbox{where} 
\;\;
\calA_0(v,\bsigma) 
:= 
\left(
\begin{array}{c} 
\partial_t v - \divv_x \bsigma \\ \partial_t\bsigma - \nabla_x v 
\end{array}
\right)
,
\end{equation}
and the data vector
\begin{equation}\label{eq:WavDatF}
F := \left(\begin{array}{c} f \\ \bg \\  v_0 \\ \bsigma_0 \end{array} \right) 
\in 
L^2(I\times\spatialdomain) \times L^2(I\times\spatialdomain)^d \times L^2(\spatialdomain) \times L^2(\spatialdomain)^d
=: \bbL(\domain)
\;.
\end{equation}
The solution space which renders $\VL$ in \eqref{eq:AcWvOps}
an isomorphism according to \cite[Thm.~3.5]{FK23} 
is
\begin{equation}\label{eq:WavVSpc}
\bbV(\domain) = V_0,
\end{equation}
where $V_0$, as defined in \cite[Sec.~3]{FK23},
is a closed subspace of
\begin{align*}
V
:= 
\big\{ \bsu = (v,\bsigma) \in L^2(I\times \spatialdomain)^{d+1}  : 
   &\, \divv(v,-\bsigma) \in L^2(\domain), \;\; 
   \partial_t \bsigma - \nabla_x v \in L^2(\domain)^d,
   \\
   &\,v(0,\cdot) \in L^2(\spatialdomain), \;\;
   \bsigma(0,\cdot) \in L^2(\spatialdomain)^d
\big\}
\;,
\end{align*}
where $\divv(v,-\bsigma) = \partial_t v - \nabla_x \bsigma$.
We endow $V$ and $\bbV(\domain) = V_0$ with the norm
\begin{align}
\nonumber
\| \bsu \|_{\bbV(\domain)}^2 
:= &\,
\| v \|^2_{L^2(\domain)}
	+ \| \partial_t\bsigma - \nabla_x v \|^2_{L^2(\domain)^d}
	+ \| \bsigma \|^2_{L^2(\domain)^d}
	+ \| \divv(v,-\bsigma) \|^2_{L^2(\domain)}
\\
&\, + \norm[L^2(\spatialdomain)]{ v(0,\cdot) }^2 
	+ \norm[L^2(\spatialdomain)^d]{ \bsigma(0,\cdot) }^2 
\;.
\label{eq:bbVwavenorm}
\end{align}
Conditions (A3)--(A4) hold and
\cite[Thm.~3.5]{FK23} states that 
$\VL$ is an isomorphism.
This isomorphism property entails well-posedness of the 
\emph{least squares formulation}, 
i.e. (A1)--(A2) hold.

\begin{proposition}
\label{prop:wave}
Given $F\in \bbL(\domain)$,
there exists a unique solution $U:= (v,\bsigma) \in \bbL(\domain)$.
This solution is, for given data $F = (f,\bg,v,\bsigma_0)$,
the unique minimizer of the LSQ functional
\begin{align}
\nonumber
(v,\bsigma) 
= \bsu
\mapsto 
\LS(\bsu;F)
&\, = 
\| \VL(v,\bsigma) - (f,\bg,v_0,\bsigma_0) \|^2_{\bbL(\domain)}
\\
\label{eq:LSWave}
&\, = \| \partial_t v - \divv_x\bsigma - f \|^2_{L^2(\domain)}
+ \| \partial_t\bsigma - \nabla_x v - \bg \|^2_{L^2(\domain)^d}
\\\nonumber
&\, \quad +\| v(0,\cdot) - v_0 \|^2_{L^2(\spatialdomain)} 
+ \| \bsigma(0,\cdot) - \bsigma_0 \|^2_{L^2(\spatialdomain)^d}
.
\end{align}
\end{proposition}
%
\subsection{Optimal Control Problems}
\label{sec:OCP}

The FoSLS formulation accommodates classical (e.g. \cite{JLLOptCtrl})
variational formulations of optimal control problems (OCPs) 
constrained by elliptic or parabolic
PDEs. 
We recapitulate the elegant and versatile 
variational setting and the corresponding 
LSQ functionals from \cite{TFMKOptCtrl2022}.
\subsubsection{Abstract OCP}
\label{sec:AbsOCP}
Consider a PDE constrained optimal control problem,
where the PDE is of the type \eqref{eq:PDE}
and satisfies (A1)--(A4).
In this section, we will denote the differential operator by $\VL_Y$.
Let
$\widehat{\bbL} = L^2(\domain'_1) \times \cdots \times L^2(\domain'_{n'})$
for $n'\in\N$ and polytopal domains $\domain'_1,\ldots,\domain'_{n'}$ satisfying (D),
and let
$\bbV_Y$ 
be such that
$\VL_Y : \bbV_Y \to \widehat{\bbL}$
is an isomorphism and such that (A3)--(A4) hold
with $\widehat{\bbL}$ instead of $\bbL(\domain)$
and $\bbV_Y$ instead of $\bbV(\domain)$,\footnote{
For brevity, we omit in this section the domain $\domain$
from the notation for function spaces.}
which implies that for all $F_Y\in\widehat{\bbL}$
conditions (A1) -- (A4) are satisfied.

In addition, we consider the space of controls
$\bbX = L^2(\domain''_1) \times \cdots \times L^2(\domain''_{{n''}})$
for $n'' \in \N$ and 
polytopal domains $\domain''_1,\ldots,\domain''_{n''}$ satisfying (D),
and 
bounded linear operators
$\VB : \bbX \to \widehat{\bbL}$
with adjoint
$\VB^* : \widehat{\bbL} \to \bbX$,
$\VA : \widehat{\bbL} \to \widehat{\bbL}$
with adjoint
$\VA^* : \widehat{\bbL} \to \widehat{\bbL}$,
$\VI_Y : \bbV_Y \to \widehat{\bbL}$
and
the self-adjoint, positive definite
$\VC : \bbX \to \bbX$.

For a closed subspace $\bbX_\ad \subset \bbX$ and given $Z\in\widehat{\bbL}$,
we consider the unconstrained optimal control problem of minimizing
\begin{align}
\label{eq:OCPcost}
J(\control) := \norm[\widehat{\bbL}]{ \VA \VI_Y Y(\control) - Z }^2 
	+ ( \VC \control , \control )_{\bbX}
\end{align}
over $\control\in\bbX_\ad$,
constrained by the PDE
\begin{align}
\label{eq:OCPpde}
\VL_Y Y = F_Y - \VB \control
,
\end{align}
of the form \eqref{eq:PDE},
with control $\control \in \bbX$
and solution $Y(\control) \in \bbV_Y$.

As in \cite[Section 2.3]{TFMKOptCtrl2022},
following \cite[Chapter 2, Theorem 1.4]{JLLOptCtrl},
we can write the OCP 
$J(\control) = \min_{V\in \bbX_\ad} J(V)$ subject to \eqref{eq:OCPpde}
equivalently as
\begin{subequations}
\label{eq:OCPpdesyst}
\begin{align}
\label{eq:OCPpdesysti}
\VL_Y Y = &\, F_Y - \VB \control,
	\\
	\label{eq:OCPpdesystii}
\VL_P P = &\, \VA^* ( \VA \VI_Y Y - Z ),
	\\
	\label{eq:OCPpdesystiii}
( - \VB^* \VI_P P + \VC \control, V - \control ) \geq &\, 0,
	\qquad 
	\text{ for all }
	V \in \bbX_\ad
	,
\end{align}
\end{subequations}
where 
$\VL_P : \bbV_P \to \widehat{\bbL}$
is an isomorphism
satisfying (A3) -- (A4)
with 
$\widehat{\bbL}$ instead of $\bbL(\domain)$
and with 
$\bbV_P$ instead of $\bbV(\domain)$ 
(it thus also satisfies (A1)--(A2) for all right-hand sides in $\widehat{\bbL}$),
and where $\VL_Y$ and $\VL_P$ are adjoint in the sense that
for a bounded linear operator
$\VI_P : \bbV_P \to \widehat{\bbL}$ holds
\begin{align}
\label{eq:adjointLyLp}
( \VL_Y Y, \VI_P P) = ( \VI_Y Y, \VL_P P)
	\qquad
	\text{ for all } 
	Y \in \bbV_Y, P\in\bbV_P
	.
\end{align}

Because $\bbX_\ad \subset \bbX$ is a closed subspace,
the variational inequality \eqref{eq:OCPpdesystiii} for the optimal control
has the solution 
$\control = \Pi_\ad \VC^{-1} \VB^* \VI_P P$,
where
$\Pi_\ad : \bbX \to \bbX_\ad$ 
is the orthogonal projection with respect to the bilinear form
$(\VC \cdot , \cdot )$,
see \cite[Section 2.5]{TFMKOptCtrl2022}.
Thus,
\eqref{eq:OCPpdesyst} reduces to 
\begin{subequations}
\label{eq:OCPpdesystunc}
\begin{align}
\label{eq:OCPpdesystunci}
\VL_Y Y = &\, F_Y - \VB \Pi_\ad \VC^{-1} \VB^* \VI_P P,
	\\
	\label{eq:OCPpdesystuncii}
\VL_P P = &\, \VA^* ( \VA \VI_Y Y - Z ).
\end{align}
\end{subequations}

This fits into the abstract setting \eqref{eq:PDE}
with the choices 
$U := ( Y, P ) \in \bbV_Y \times \bbV_P =: \bbV$,
$\bbL := \widehat{\bbL} \times \widehat{\bbL}$
and 
$\VL:\bbV\to\bbL$
such that
\begin{align}
\label{eq:FoSLSOCP}
\VL U 
:=
\VL 
\left(
\begin{array}{c} 
Y
\\
P
\end{array}
\right)
:=
\left(
\begin{array}{c} 
\VL_Y Y + \VB \Pi_\ad \VC^{-1} \VB^* \VI_P P
\\
\VL_P P - \VA^* \VA \VI_Y Y
\end{array}
\right)
=
\left(
\begin{array}{c} 
F_Y
\\
- \VA^* Z
\end{array}
\right)
=: F
\in \bbL
.
\end{align}
\begin{proposition}[{{\cite[Theorem 7]{TFMKOptCtrl2022}}}]
\label{prop:FoSLSOCP}
The FoSLS formulation \eqref{eq:FoSLSOCP}
satisfies Assumptions (A1), (A3)--(A4).
\end{proposition}
\begin{proof}
Assumption (A1) holds for $\VL$ by \cite[Theorem 7]{TFMKOptCtrl2022}.
Assumptions (A3)--(A4) hold for $\VL$,
because we assumed that they hold for $\VL_Y$, $\VL_P$.
\end{proof}
With the corresponding LSQ functional
\begin{align}
\label{eq:LSQLOCP}
\LS(U;F)
:= &\,
\| \VL_Y Y + \VB \Pi_\ad \VC^{-1} \VB^* \VI_P P - F_Y \|^2_{\widehat{\bbL}}
+
\| \VL_P P - \VA^* ( \VA \VI_Y Y - Z ) \|^2_{\widehat{\bbL}} 
,
\end{align}
by (A1), there exists a unique 
least squares solution
$U \in \bbV(\domain)$
to \eqref{eq:FoSLSMin}, see also \cite[Theorem 8]{TFMKOptCtrl2022}.
If (A2) is not satisfied, 
the residual of the solution 
$\norm[\bbL(\domain)]{ F - \VL U }$
is strictly positive.
The error-residual relation \eqref{eq:PDEUniq} follows directly from (A1)
and holds regardless of (A2).

\begin{remark}
\label{rem:ocpa2}

Because ${\rm range}(\VL) \subset \bbL$ is a closed subspace,
for all $F_Y \in \widehat{\bbL}$ and $Z\in\widehat{\bbL}$
we obtain the orthogonal decomposition
$F = ( F_Y, - \VA^* Z ) = \Pi_{{\rm range}(\VL)} F + ( F - \Pi_{{\rm range}(\VL)} F )$.
Similarly, 
for all $v \in \bbV$ 
the LSQ functional from \eqref{eq:LSQLOCP}
can be decomposed into
\begin{align}
\label{eq:LSQLdecomp}
\LS(v;F)
= &\,
\LS(v;\Pi_{{\rm range}(\VL)} F)
+
\| \Pi_{{\rm range}(\VL)} F - F \|^2_{\bbL}
.
\end{align}
From this we see that 
\begin{equation}
\label{eq:FoSLSMinPiF}
U 
= {\rm arg}\min_{v\in \bbV(\domain)} 
\LS(v;F)
= {\rm arg}\min_{v\in \bbV(\domain)} 
\LS(v;\Pi_{{\rm range}(\VL)} F)
,
\end{equation}
where $\LS(v;\Pi_{{\rm range}(\VL)} F)$ satisfies (A1)--(A4).

\end{remark}

\subsubsection{OCP for the Poisson Equation}
\label{sec:OCPPoisson}

For a bounded, polytopal Lipschitz domain $\domain \subset \R^d$ satisfying (D), 
for some $d\in\N$,
and
given $z \in L^2(\domain)$ and $\lambda>0$,
we consider the OCP 
$\min_{q\in L^2(\domain)} \norm[L^2(\domain)]{ u - z }^2 + \lambda \norm[L^2(\domain)]{ q }^2$
constrained by the Poisson equation 
$- \Delta u = f + q$,
for the right-hand side $f\in L^2(\domain)$
and the control $q \in L^2(\domain)$,
with solution $u \in H^1_0(\domain)$.
%
The following discussion is analogous to that in \cite[Section 3.1]{TFMKOptCtrl2022}.
We use the notation of \cite[Section 3.1]{FuePraeLSQ},
as introduced in Section \ref{sec:Poisson},
which differs slightly from that in \cite[Section 3.1]{TFMKOptCtrl2022}.

With
$\widehat{\bbL} = L^2(\domain)^{d+1}$,
$Z := ( z, \bszero ) \in \widehat{\bbL}$
and
$\bbV_Y = H^1_0(\domain) \times H(\divv;\domain)$,
let 
$\VL_Y : \bbV_Y \to \widehat{\bbL}$ 
and 
$F_Y = ( f, \bszero ) \in \widehat{\bbL}$ 
be the first order differential operator and right-hand side
defined in \eqref{eq:PoissonFoSLS}.
Then, \eqref{eq:adjointLyLp} is satisfied
with
$\VI_Y : \bbV_Y \to \widehat{\bbL} : (u, \bsigma) \mapsto (u, \bsigma)$
and
$\bbV_P = \bbV_Y$,
$\VL_P = \VL_Y$, 
$\VI_P = \VI_Y$.
In addition, let
$\VA : \widehat{\bbL} \to \widehat{\bbL} : ( v, \bstau ) \mapsto ( v, \bszero )$,
$\bbX_\ad := \bbX := L^2(\domain)$,
$X := q$,
$\Pi_\ad = \Id_\bbX$,
$\VB : \bbX \to \widehat{\bbL} : w \mapsto - ( w, \bszero )$
and
$\VC : \bbX \to \bbX : w \mapsto \lambda w$.
%
Then, the OCP introduced at the beginning of this subsection 
is equivalent to \eqref{eq:FoSLSOCP}.

\subsubsection{OCP for Maxwell Equations}
\label{sec:OCPMaxwell}
For a bounded, polytopal Lipschitz domain $\domain \subset \R^3$
satisfying condition (D)
and
given $\bsz \in L^2(\domain)^3$ and $\lambda>0$,
we consider the OCP 
$\min_{\bsq\in L^2(\domain)^3} \norm[L^2(\domain)^3]{ \bsu - \bsz }^2 
	+ \lambda \norm[L^2(\domain)^3]{ \bsq }^2$
constrained by the Maxwell equations 
$\curl \bsigma + c \bsu = \bsf + \bsq$
and
$\curl \bsu - \bsigma = \bszero $,
for the right-hand side $\bsf \in L^2(\domain)^3$
and the control $\bsq \in L^2(\domain)^3$,
with solution $(\bsu,\bsigma) \in H_0(\curl;\domain) \times H(\curl;\domain)$.
The following discussion is analogous to that in \cite[Section 3.3]{TFMKOptCtrl2022},
using the notation of \cite[Section 3.4]{FuePraeLSQ}
as introduced in Section \ref{sec:nLSQCEM}.

With
$\widehat{\bbL} = L^2(\domain)^6$,
$Z := ( \bsz, \bszero) \in \widehat{\bbL}$
and
$\bbV_Y = H_0(\curl;\domain) \times H(\curl;\domain)$,
let 
$\VL_Y : \bbV_Y \to \widehat{\bbL}$ 
and 
$F_Y = ( \bsf, \bszero ) \in \widehat{\bbL}$ 
be the first order differential operator and right-hand side
defined in \eqref{eq:FoSLSMaxwell}.
Then, \eqref{eq:adjointLyLp} is satisfied
with
$\VI_Y : \bbV_Y \to \widehat{\bbL} : (\bsu, \bsigma) \mapsto (\bsu, \bsigma)$
and
$\bbV_P = \bbV_Y$,
$\VL_P = \VL_Y$, 
$\VI_P = \VI_Y$.
In addition, let
$\VA : \widehat{\bbL} \to \widehat{\bbL} : ( \bsv, \bstau ) \mapsto ( \bsv, \bszero )$,
$\bbX_\ad := \bbX := L^2(\domain)^3$,
$X := \bsq$,
$\Pi_\ad = \Id_\bbX$,
$\VB : \bbX \to \widehat{\bbL} : \bsw \mapsto - ( \bsw, \bszero )$
and
$\VC : \bbX \to \bbX : \bsw \mapsto \lambda \bsw$.
%
Then, the OCP introduced at the beginning of this subsection 
is equivalent to \eqref{eq:FoSLSOCP}.

\subsubsection{OCP for the Heat Equation}
\label{sec:OCPSpacetime}
In this section, we introduce the optimal control problem for the heat equation
as considered in \cite[Section 3.4]{TFMKOptCtrl2022},
which is a special case of \eqref{eq:Parabolic} from Section \ref{sec:nLSQParab},
as we fix $\bsa \in \R^{d\times d}$ to be the identity matrix, 
$\bsb=\bszero$ and $c = 0$.
For $\domain = I \times \spatialdomain$, $I = ( 0, T )$, $T>0$
and a bounded, polytopal Lipschitz domain $\spatialdomain \subset \R^d$
for some $d\in\N$,
satisfying (D), 
with boundary $\Gamma := \partial\spatialdomain$,
given $z \in L^2(\domain)$, $z_T \in L^2(\spatialdomain)$, $\lambda,\lambda_0>0$,
we consider the OCP 
$\min_{q\in L^2(\domain), q_0\in L^2(\spatialdomain)} 
\norm[L^2(\domain)]{ u - z }^2 + \norm[L^2(\spatialdomain)]{ u(T,\cdot) - z_T }^2 
	   + \lambda \norm[L^2(\domain)]{ q }^2 + \lambda_0 \norm[L^2(\spatialdomain)]{ q_0 }^2$
constrained by 
\begin{equation} \label{eq:ParabolicOCP}
\begin{array}{rcl} 
\partial_t u - \divv_x( \nabla_x u ) & = & f + q
\quad \mbox{in}\quad I\times \spatialdomain
\\
u & = & 0 
\quad \mbox{on} \quad I\times \Gamma,
\\
u(0,\cdot) & = & u_0 + q_0
\quad \mbox{in} \quad \spatialdomain,
\end{array}
\end{equation}
for the right-hand side $f\in L^2(\domain)$,
the initial value $u_0\in L^2(\spatialdomain)$
and the controls $q \in L^2(\domain)$ 
and $q_0 \in L^2(\spatialdomain)$,
with solution 
$u \in L^2(I;H^1_0(\spatialdomain)) \cap H^1(I;H^{-1}(\spatialdomain))$.
%
The following discussion is similar to that in \cite[Section 3.4]{TFMKOptCtrl2022}.

With
$\widehat{\bbL} = L^2(\domain)^d \times L^2(\domain) \times L^2(\spatialdomain)$
and given $Z := ( \bszero, z, z_T ) \in \widehat{\bbL}$,
for
\begin{align*}
\bbV_Y = \{\bsu\in L^2(I;H^1_0(\spatialdomain))\times L^2(I\times \spatialdomain)^d: \divv\bsu\in L^2(I\times \spatialdomain)\}
	,
\end{align*}
with norm\footnote{This norm is equivalent to 
the norm on $\bbV$ introduced in Section \ref{sec:nLSQParab}.
The arguments from Footnote \ref{fn:tracetimezero}
also apply here, the new term can be treated in the same way as 
$\norm[L^2(\spatialdomain)]{ u_1(0,\cdot) }^2$.
We will need the additional term in Remark \ref{rem:plainconvOCP} 
to show local boundedness (L) of the map $\VI_Y$ defined below.
For details see that remark.}
\begin{align}
\nonumber
\| \bsu \|_{\bbV_Y}^2 
:= &\, 
\| u_1 \|^2_{L^2(I;H^1(\spatialdomain))} 
	+ \| \bsu_2 \|^2_{L^2(I;L^2(\spatialdomain)^d)} 
	+ \| \divv \bsu \|^2_{L^2(I\times\spatialdomain)}
\\
&\, + \norm[L^2(\spatialdomain)]{ u_1(0,\cdot) }^2 
	+ \norm[L^2(\spatialdomain)]{ u_1(T,\cdot) }^2 
,
\label{eq:bbVYnorm}
\end{align}
let 
$\VL_Y : \bbV_Y \to \widehat{\bbL}$ 
be defined by
\begin{equation}
\label{eq:Heat1stOrd}
\VL_Y U
:=
\VL_Y
\left(\begin{array}{c} 
u_1
\\
\bsu_2
\end{array}
\right)
:= 
\left(
\begin{array}{c} 
\bsu_2 + \nabla_x u_1
\\
\partial_t u_1 +\divv_x \bsu_2 
\\
u_1(0,\cdot)
\end{array}
\right)
\end{equation}
and let 
$F_Y = ( \bszero, f, u_0 ) \in \widehat{\bbL}$ 
be the 
right-hand side
defined in \eqref{eq:Para1stOrd}.
Then, \eqref{eq:adjointLyLp} is satisfied
with
$\VI_Y : \bbV_Y \to \widehat{\bbL} : ( u_1, \bsu_2 ) \mapsto ( \bsu_2, u_1, u_1(T,\cdot) )$,
$\bbV_P = \bbV_Y$,
$\VL_P : \bbV_P \to \widehat{\bbL} $
defined by
\begin{equation}\label{eq:Para1stOrdAdj}
\VL_P P
=
\VL_P 
\left(\begin{array}{c} 
p_1
\\
\bsp_2
\end{array}
\right)
:=
\left(
\begin{array}{c} 
\bsp_2 - \nabla_x p_1
\\
-\partial_t p_1 - \divv_x \bsp_2
\\
p_1(T,\cdot)
\end{array}
\right)
\end{equation}
and
$\VI_P :  ( p_1, \bsp_2) \mapsto ( \bsp_2, p_1, p_1(0,\cdot) )$.
In addition, let
$\VA : \widehat{\bbL} \to \widehat{\bbL} : ( \bsv_1, v_2, v_3 ) \mapsto ( \bszero, v_2, v_3 )$,
$\bbX_\ad := \bbX := L^2(\domain) \times L^2(\spatialdomain)$,
$X := ( q, q_0 )$,
$\Pi_\ad = \Id_\bbX$,
$\VB : \bbX \to \widehat{\bbL} : ( w, w_0 ) \mapsto - ( \bszero, w, w_0 )$
and
$\VC : \bbX \to \bbX : ( w, w_0 ) \mapsto ( \lambda w, \lambda_0 w_0 )$.
%
Then, the OCP introduced at the beginning of this subsection 
is equivalent to \eqref{eq:FoSLSOCP}.

The case $q_0 = 0$ can also be treated, and corresponds to taking
$\bbX_\ad := \bbX := L^2(\domain)$, 
$X := q$,
$\VB : \bbX \to \widehat{\bbL} : w \mapsto - ( \bszero, w, 0 )$
and 
$\VC : \bbX \to \bbX : w \mapsto \lambda w$ for some $\lambda>0$.
Also in this case, the OCP is equivalent to \eqref{eq:FoSLSOCP}.

\section{Least Squares Approximation}
\label{sec:LSQFEM}
Based on the LSQ formulation of 
PDE initial boundary value problems in Section~\ref{sec:LSQForm}, 
stable discretizations are readily obtained
by constraining the minimization in \eqref{eq:FoSLSMin} to  
finite-dimensional subspaces 
$\bbV_\bullet(\domain)$ of $\bbV(\domain)$.
Eventually, in the next section, we choose $\bbV_\bullet(\domain)$ 
as one of the NN emulated, finite-dimensional spaces constructed in \cite{LODSZ22_991}.

The \emph{Least Squares Galerkin Method} (``LSQ-G'' for short) 
seeks
\begin{equation}\label{eq:FoSLSMinN}
U_{\bullet} 
= 
{\rm arg}\min_{v_{\bullet} \in \bbV_{\bullet}(\domain)} 
	\| F - \VL v_{\bullet} \|^2_{\bbL(\domain)}
\;.
\end{equation}
Evidently, 
for \emph{linear operators $\VL$}, and 
for \emph{linear subspaces $\bbV_{\bullet}(\domain)$}, 
the functional 
$v_{\bullet}\mapsto \| F - \VL v_{\bullet} \|^2_{\bbL(\domain)}$ 
is a quadratic functional
of the coefficients in representations of $v_{\bullet}$ 
in terms of a basis of the (linear) subspace $\bbV_\bullet$.
There exists a unique minimizer of \eqref{eq:FoSLSMinN} 
which is quasi-optimal by \eqref{eq:PDEUniq} 
\be\label{eq:LSQQO}
c_\VL \| U - U_\bullet \|_{\bbV(\domain)}
	\leq \| F - \VL U_\bullet \|_{\bbL(\domain)}
	= \min_{v_\bullet \in \bbV_\bullet(\domain)} \| F - \VL v_\bullet \|_{\bbL(\domain)}
	\leq C_\VL \min_{v_\bullet \in \bbV_\bullet(\domain)} \| U - v_\bullet \|_{\bbV(\domain)}
.
\ee

For the examples from Sections \ref{sec:Source}--\ref{sec:OCP},
we will use lowest order FE spaces in the de Rham complex
on regular, simplical partitions $\calT$ 
of polytopal domains $\domain \subset \bbR^d$.
First, we recall from \cite[Section 1.4.1]{LODSZ22_991}
notation for such partitions.
For
$k\in\{0,\ldots,d\}$ we define a $ k $-simplex $ K $ by
$ K = \conv(\{a_0,\ldots, a_k\}) \subset \R^d $, for some
$ a_0,\ldots,a_k \in \R^d $ which do not all lie in one affine
subspace of dimension $k-1$, and where
\begin{equation*}
\conv (Y) 
:= 
\setc{x = \sum_{y \in Y}\lambda_y y}{ \lambda_y\geq0 \, \text{ and }\, \sum_{y \in Y} \lambda_y = 1 }
\end{equation*}
denotes the closed convex hull.\footnote{
In \cite{LODSZ22_991}, 
$\conv$ denotes the \emph{open convex hull}.
As a result, in \cite{LODSZ22_991} simplices are open by definition.
Here, simplices are closed by definition.}
By $\snorm{T}$ we will denote the
$k$-dimensional Lebesgue measure of a $k$-simplex.  We consider a
simplicial partition $ \calT $ on $ \domain $ of $d$-simplices,
i.e. 
$ \overline{\domain} = \bigcup_{K\in \calT} K $ and
$ \interior K\cap \interior K' = \emptyset$, for all $ K\neq K' $.
We assume that $ \calT $ is a \emph{regular} partition, 
i.e. for all distinct $K, K' \in\calT$ 
it holds that $K\cap K'$ 
is a $k$-subsimplex of $K$ for some $k\in\{0,\ldots,d-1\}$.
I.e., 
there exist
$a_0,\ldots,a_d \in \domain$ 
such that $K = \conv(\{a_0,\ldots,a_d\})$
and
$K\cap K' = \conv(\{a_0,\ldots,a_k\})$.
%
%
Let $ \calV $ be the set of vertices of $ \calT $.  We also let
$ \calF,\calE $ be the sets of $ (d-1) $- and $ 1 $-subsimplices of
$ \calT $, whose elements are called \emph{faces} and \emph{edges},
respectively, that is
\begin{align*}
  \calF &:= \set{f \subset \overline{\domain} }{ \exists K = \conv(\{a_0,\ldots, a_d\})\in \calT, \exists i\in\{0,\ldots,d\}
          \text{ with } f = \conv(\{a_0,\ldots, a_{d}\}\backslash \{a_i\} )}, \\
  \calE &:=\set{e \subset \overline{\domain}}{\exists K = \conv(\{a_0,\ldots, a_d\}) \in \calT, \exists i,j\in\{0,\ldots,d\}, i\neq j,
          \text{ with } e = \conv(\{a_i,a_j \})}.
\end{align*}
%
For $\calI \in \{ \calV, \calE, \calF, \calT \}$,
for all $i \in \calI$ we denote by
$s(i) := \# \{ K \in \calT : i \subset K \}$
the number of elements of $\calT$ sharing the subsimplex $i$,
and define
$\mathfrak{s}(\calI) := \max_{i\in\calI} s(i)$.

For the example from Section \ref{sec:Poisson},
we can use the subspace
\begin{align}
\label{eq:bbVPoisson}
\bbV_\bullet(\domain) 
	= &\, 	( \So(\calT) \cap H^1_0(\domain) ) \times \RT(\calT) 
	\subset H^1_0(\domain) \times H(\divv;\domain) 
	= \bbV(\domain)
	.
\end{align}
In \cite{BM2023}, under additional assumptions,
optimal convergence rates as $h(\cT)\to 0$ were proved.
For
the example from Section \ref{sec:Elasticity} we can argue similarly
and use the subspace
\begin{align}
\label{eq:bbVElasticity}
\bbV_\bullet(\domain) 
	= &\, 	( \So(\calT) \cap H^1_0(\domain) )^d \times \RT(\calT)^d 
	\subset H^1_0(\domain)^d \times H(\divv;\domain)^d 
	= \bbV(\domain)
	.
\end{align}
For the example from Section \ref{sec:nLSQCEM}, 
we can consider
\begin{align}
\label{eq:bbVMaxw}
\bbV_\bullet(\domain) 
	= &\, ( \Ne(\calT) \cap H_0(\curl;\domain) ) \times \Ne(\calT) 
	\subset H_0(\curl;\domain) \times H(\curl;\domain)
	= \bbV(\domain)
	.
\end{align}
For the example from Section \ref{sec:nLSQParab},
for any regular, simplicial triangulation $\calT$ of 
the space-time cylinder $\domain = I \times \spatialdomain$,
we can use the subspace
\begin{align}
\label{eq:bbVParab}
\bbV_\bullet(\domain)
= &\, ( \So(\calT) \cap L^2( I; H^1_0( \spatialdomain ) ) ) \times \So(\calT)^d
	\\
	\subset &\, \{\bsu\in L^2(I;H^1_0(\spatialdomain))\times L^2(I\times \spatialdomain)^d: 
		\divv\bsu\in L^2(I\times \spatialdomain) \}
	=
	\bbV(\domain)
.
\nonumber
\end{align}
In this example,
the restriction of the $\bbV$-norm to a simplex $K \in \calT$
needs special attention. 
As in Step 1 in the proof of \cite[Theorem 3.3]{GaStLSQ},
we define for all $K \in \calT$ and all $\bsu = ( u_1, \bsu_2 ) \in \bbV(\domain)$
\begin{align}
\label{eq:VKtracezeroParab}
\| \bsu \|_{\bbV(K)}^2 
:= 
\| u_1 \|^2_{L^2(K)}
	+ \| \nabla_x u_1 \|^2_{L^2(K)^d} 
	+ \| \bsu_2 \|^2_{L^2(K)^d} 
	+ \| \divv \bsu \|^2_{L^2(K)}
	+ \norm[L^2(\partial_0 K)]{ u_1(0,\cdot)|_{\partial_0 K} }^2
\;,
\end{align}
where $\partial_0 K = K \cap ( \{0\} \times \Omega )$.
For the example from Section \ref{sec:nLSQWave},
we can use the same spacetime finite element space 
defined in
\eqref{eq:bbVParab}
(see \cite[Section 4.1]{FK23}).
Again,
the restriction of the $\bbV$-norm to a simplex $K \in \calT$
needs special attention. 
We define for all $K \in \calT$ and all $\bsu = ( v, \bsigma ) \in \bbV(\domain)$
\begin{align}
\nonumber
\| \bsu \|_{\bbV(K)}^2 
:= &\, \| v \|^2_{L^2(K)}
	+ \| \partial_t\bsigma - \nabla_x v \|^2_{L^2(K)^d}
	+ \| \bsigma \|^2_{L^2(K)^d}
	+ \| \divv(v,-\bsigma) \|^2_{L^2(K)}
\\
&\, + \norm[L^2(\partial_0 K)]{ v(0,\cdot)|_{\partial_0 K} }^2 
	+ \norm[L^2(\partial_0 K)^d]{ \bsigma(0,\cdot)|_{\partial_0 K} }^2 
\;,
\label{eq:VKtracezeroWave}
\end{align}
where $\partial_0 K = K \cap ( \{0\} \times \Omega )$.

By \cite[Theorem 8]{TFMKOptCtrl2022},
the optimal control problems from Section \ref{sec:OCP},
although they do not satisfy (A2),
do admit a unique Galerkin solution to \eqref{eq:FoSLSMinN}
which satisfies \eqref{eq:LSQQO}.
For the OCP from Section \ref{sec:OCPPoisson}, 
constrained by the Poisson equation,
we can discretize 
$\bbV_Y = \bbV_P = H^1_0(\domain) \times H(\divv;\domain)$ 
as in \eqref{eq:bbVPoisson}.
For the OCP from Section \ref{sec:OCPMaxwell},
constrained by the Maxwell equations,
we can discretize
$\bbV_Y = \bbV_P = H_0(\curl;\domain) \times H(\curl;\domain)$
as in \eqref{eq:bbVMaxw}.
For the heat equation OCP from Section \ref{sec:OCPSpacetime},
we can discretize 
$\bbV_Y = \bbV_P$
as in \eqref{eq:bbVParab}.
Special attention to the norm on $\bbV_Y = \bbV_P$ is necessary,
because the $\bbV_Y$-norm defined in \eqref{eq:bbVYnorm}
has the additional term $\norm[L^2(\spatialdomain)]{ u_1(T,\cdot) }^2$
when compared to the $\bbV$-norm defined in \eqref{eq:bbVheatnorm}.
In particular, we define 
for all $K \in \calT$ and all $\bsu = ( u_1, \bsu_2 ) \in \bbV_Y(\domain)$
\begin{align}
\| \bsu \|_{\bbV_Y(K)}^2 
:= &\,
\| u_1 \|^2_{L^2(K)}
	+ \| \nabla_x u_1 \|^2_{L^2(K)^d} 
	+ \| \bsu_2 \|^2_{L^2(K)^d} 
	+ \| \divv \bsu \|^2_{L^2(K)}
	+ \norm[L^2(\partial_0 K)]{ u_1(0,\cdot)|_{\partial_0 K} }^2
	\nonumber\\
	& + \norm[L^2(\partial_T K)]{ u_1(T,\cdot)|_{\partial_T K} }^2
\;,
\label{eq:VYKtracezeroT}
\end{align}
where $\partial_0 K = K \cap ( \{0\} \times \Omega )$
and $\partial_T K = K \cap ( \{T\} \times \Omega )$.

\section{Deep Least Squares}
\label{sec:NeuLSQ}
%
With the generic LSQ formulation \eqref{eq:FoSLSMinN},
the deep LSQ setting is to minimize
the loss function generated by the LSQ residual over 
realizations of NNs from a suitable set of admissible NNs.

%
\subsection{Neural Network Definitions}
\label{sec:nndef}
In order to present the deepLSQ method for the FoSLS formulations
in the previous section,
we recapitulate basic terminology and NN definitions from \cite{LODSZ22_991}.
\begin{definition}[{{\cite[Section 2.1]{LODSZ22_991}}}]
\label{def:nn} For $d,L\in\N$, a \emph{neural network $\Phi$} 
  with input dimension $d \geq 1$ and depth (i.e., number of layers) $L\geq 1$, 
  comprises a finite collection of activation
  functions\footnote{No activation is applied in the output layer $L$.
  We introduce $\varrho_{L}$ only for consistency of notation,
  and define it to be equal to the identity function.}
  $\bsvarrho =
  \{\varrho_{\ell}\}_{\ell=1}^L$
  and a finite sequence of matrix-vector tuples, i.e.
  \begin{align*}
    \Phi = ((A_1,b_1,\varrho_1),(A_2,b_2,\varrho_2),\ldots,(A_L,b_L,\varrho_L)).
  \end{align*}

  For $N_0 := d$ and \emph{numbers of neurons $N_1,\ldots,N_L\in\N$
    per layer}, for all $\ell=1,\ldots, L$ it holds that
  $A_\ell\in\bbR^{N_\ell \times N_{\ell-1} }$ and
  $b_\ell\in\bbR^{N_\ell}$, and that $\varrho_\ell$ is a list of
  length $N_\ell$ of \emph{activation functions}
  $(\varrho_\ell)_i: \R\to\R$, $i=1,\ldots,N_\ell$, acting on node $i$
  in layer $\ell$.

  The \emph{realization} of $\Phi$ is the function
  \begin{align*}
    \realiz{\Phi}: \R^d\to\R^{N_L} : x \to x_L,
  \end{align*}
  where
  \[
    x_0  := x, \;\; 
    x_\ell := \varrho_\ell( A_\ell x_{\ell-1} + b_\ell ),
             \qquad\text{ for } \ell=1,\ldots,L.
  \]
  Here, for
  $\ell=1,\ldots,L$, 
  the activation function $\varrho_\ell$ is applied component\-wise:
  for $y = (y_1,\ldots,y_{N_\ell})\in\bbR^{N_\ell}$ we
  denote
  $\varrho_\ell(y) = ( (\varrho_\ell)_1(y_1), \ldots,
  (\varrho_\ell)_{N_\ell}(y_{N_\ell}) )$, i.e., $(\varrho_\ell)_i$ 
is the activation function applied in position $i$ of layer $\ell$.

  We call the layers indexed by $\ell=1,\ldots,L-1$ \emph{hidden
    layers}, in those layers activation functions are applied. 
  We fix the activation function in the last layer of the NN to be the identity, 
  i.e., $\varrho_L := \Id_{\R^{N_L}}$.

  We refer to $\depth(\Phi) := L$ as the \emph{depth} of $\Phi$.  
  For
  $\ell=1,\ldots,L$ we denote by
  $\size_\ell(\Phi) := \norm[0]{A_\ell} + \norm[0]{b_\ell} $ the
  \emph{size of layer $\ell$}, which is the number of nonzero
  components in the weight matrix $A_\ell$ and the bias vector
  $b_\ell$, and call $\size(\Phi) := \sum_{\ell=1}^L \size_\ell(\Phi)$
  the \emph{size} of $\Phi$.
  Furthermore,
  we call $d$ and $N_L$ the
  \emph{input dimension} and the \emph{output dimension}, and denote
  by $\sizefirst(\Phi) := \size_1(\Phi)$ and
  $\sizelast(\Phi) := \size_L(\Phi)$ the size of the first and the
  last layer, respectively.
\end{definition}
Our NNs will use two different activation functions. 
First, we
use the \emph{Rectified Linear Unit} (\emph{ReLU}) activation
\begin{equation}
\label{eq:reludef}
    \relu(x) = \max\{ 0, x \}.
\end{equation}
In general,
NNs which only contain ReLU activations realize continuous,
piecewise linear functions. By \emph{ReLU NNs} we refer to NNs which
only have ReLU activations, including NNs of depth $1$, which do
not have hidden layers and realize affine transformations.
Second, 
for the emulation of discontinuous functions, we in addition
use the \emph{Binary Step Unit} (\emph{BiSU}) activation
\begin{align}
  \label{eq:heavidef}
  \heavi(x) =
  \begin{cases}
    0 &\text{if } x\le0,
    \\
    1 &\text{if } x>0,
  \end{cases}
\end{align}
which is also called \emph{Heaviside} function.  \emph{BiSU NNs} are
defined analogously to ReLU NNs.
%

In the following sections, 
we construct NNs from smaller NNs
using a \emph{calculus of NNs}, elements of
which we now recall from \cite{PV2018}.
The results cited from \cite{PV2018} were derived for
NNs which only use the ReLU activation function, but they also hold
for NNs with multiple activation functions without modification.
\begin{proposition}[Parallelization of NNs {{\cite[Definition
      2.7]{PV2018}}}]
  \label{prop:parallel}
  For $d,L\in\N$ let 
\[
\Phi^1 = 
  \left
    ((A^{(1)}_1,b^{(1)}_1,\varrho^{(1)}_1),\ldots,(A^{(1)}_L,b^{(1)}_L,\varrho^{(1)}_L)\right) 
\; \mbox{ and } \;
  \Phi^2 = \left
    ((A^{(2)}_1,b^{(2)}_1,\varrho^{(2)}_1),\ldots,(A^{(2)}_L,b^{(2)}_L,\varrho^{(2)}_L)\right
  ) 
\]
be two NNs with input dimension $d$ and depth $L$.  
Let the
\emph{parallelization} $\Parallel{\Phi^1,\Phi^2}$ of $\Phi^1$ and $\Phi^2$ 
be defined by
  \begin{align*}
    \Parallel{\Phi^1,\Phi^2} := &\, ((A_1,b_1,\varrho_1),\ldots,(A_L,b_L,\varrho_L)),
    &&
    \\
    A_1 = &\, \begin{pmatrix} A^{(1)}_1 \\ A^{(2)}_1 \end{pmatrix},
    \quad
    A_\ell = \begin{pmatrix} A^{(1)}_\ell & 0\\0&A^{(2)}_\ell \end{pmatrix},
    &&
       \text{ for } \ell = 2,\ldots L,
    \\
    b_\ell = &\, \begin{pmatrix} b^{(1)}_\ell \\ b^{(2)}_\ell \end{pmatrix},
    \quad
    \varrho_\ell = \begin{pmatrix} \varrho^{(1)}_\ell \\ \varrho^{(2)}_\ell \end{pmatrix},
                                &&
                                   \text{ for } \ell = 1,\ldots L.
  \end{align*}

  Then,
  \begin{align*}
    \realiz{\Parallel{\Phi^1,\Phi^2}} (x)
    = &\, ( \realiz{\Phi^1}(x), \realiz{\Phi^2}(x) ),
        \quad
        \text{ for all } x\in\R^d,
    \\
    \depth(\Parallel{\Phi^1,\Phi^2}) = L, &
                                            \qquad
                                            \size(\Parallel{\Phi^1,\Phi^2}) = \size(\Phi^1) + \size(\Phi^2)
                                                  .
  \end{align*}

\end{proposition}
The parallelization of more than two NNs is handled by repeated
application of Proposition \ref{prop:parallel}.
Similarly, we can emulate sums of realizations of NNs.

\begin{proposition}[Sum of NNs]
\label{prop:sum}
For $d,N,L\in \N$ let
\[
\Phi^1 = 
  \left
    ((A^{(1)}_1,b^{(1)}_1,\varrho^{(1)}_1),\ldots,(A^{(1)}_L,b^{(1)}_L,\varrho^{(1)}_L)\right) 
\; \mbox{ and } \;
  \Phi^2 = \left
    ((A^{(2)}_1,b^{(2)}_1,\varrho^{(2)}_1),\ldots,(A^{(2)}_L,b^{(2)}_L,\varrho^{(2)}_L)\right
  ) 
\]
be two NNs with input dimension $d$, output dimension $ N $ and
depth $L$.  Let the \emph{sum} $\Phi^1+\Phi^2$ of $\Phi^1$ and
$\Phi^2$ be defined by
\begin{align*}
\Phi^1 + \Phi^2 := &\, ((A_1,b_1,\varrho_1),\ldots,(A_L,b_L,\varrho_L)),
&&
\\
A_1 =& \begin{pmatrix} A^{(1)}_1 \\ A^{(2)}_1 \end{pmatrix},
\quad
b_1 = \begin{pmatrix} b^{(1)}_1 \\ b^{(2)}_1 \end{pmatrix},
\quad
\varrho_1 = \begin{pmatrix} \varrho^{(1)}_1 \\ \varrho^{(2)}_1 \end{pmatrix},
\\
A_\ell =& \begin{pmatrix} A^{(1)}_\ell & 0\\0&A^{(2)}_\ell \end{pmatrix},
\quad
b_\ell =  \begin{pmatrix} b^{(1)}_\ell \\ b^{(2)}_\ell \end{pmatrix},
\quad
\varrho_\ell = \begin{pmatrix} \varrho^{(1)}_\ell \\ \varrho^{(2)}_\ell \end{pmatrix},
&&
\text{ for } \ell = 2,\ldots L-1.
\\
A_L =& \begin{pmatrix} A^{(1)}_L & A^{(2)}_L \end{pmatrix},
\quad
b_L = b^{(1)}_L + b^{(2)}_L, 
\quad
\varrho_L = \Id_{\R^N}.
\end{align*}

Then,
\begin{align*}
\realiz{\Phi^1+\Phi^2} (x)
= &\, \realiz{\Phi^1}(x) + \realiz{\Phi^2}(x),
\quad
\text{ for all } x\in\R^d,
\\
\depth(\Phi^1+\Phi^2) = &\, L, \qquad
\size(\Phi^1+\Phi^2)
\leq \size(\Phi^1) + \size(\Phi^2) 
.
\end{align*}
\end{proposition}

Finally, we define the concatenation of two NNs.
\begin{definition}[Concatenation of NNs {{\cite[Definition 2.2]{PV2018}}}]
\label{def:pvconc}
For $L^{(1)}, L^{(2)} \in\N$, let 
\[
\Phi^1 = 
  \left
    ((A^{(1)}_1,b^{(1)}_1,\varrho^{(1)}_1),\ldots,(A^{(1)}_L,b^{(1)}_L,\varrho^{(1)}_L)\right) 
\; \mbox{ and } \;
  \Phi^2 = \left
    ((A^{(2)}_1,b^{(2)}_1,\varrho^{(2)}_1),\ldots,(A^{(2)}_L,b^{(2)}_L,\varrho^{(2)}_L)\right
  ) 
\]
be two NNs 
such that the input dimension of $\Phi^1$, 
which we will denote by $k$, 
equals the output dimension of $\Phi^2$.
Then, the \emph{concatenation} of $\Phi^1$ and $\Phi^2$ 
is the NN of depth $L := L^{(1)} + L^{(2)} -1$ defined as
\begin{align*}
    \Phi^1 \bullet \Phi^2 := &\, ((A_1,b_1,\varrho_1),\ldots,(A_L,b_L,\varrho_L)),
    \\
    (A_\ell,b_\ell,\varrho_\ell) = &\, (A^{(2)}_\ell,b^{(2)}_\ell,\varrho^{(2)}_\ell),
	\qquad
    \text{ for }\ell=1,\ldots,L^{(2)}-1,
    \\
    A_{L^{(2)}} = &\, A^{(1)}_1 A^{(2)}_{L^{(2)}},
    \qquad
    b_{L^{(2)}} = A^{(1)}_1 b^{(2)}_{L^{(2)}} + b^{(1)}_1,
    \qquad
    \varrho_{L^{(2)}} = \varrho^{(1)}_1,
    \\
    (A_\ell,b_\ell,\varrho_\ell) 
    	= &\, (A^{(1)}_{\ell-L^{(2)}+1},b^{(1)}_{\ell-L^{(2)}+1}, \varrho^{(1)}_{\ell-L^{(2)}+1}),
	\qquad
	\text{ for }\ell=L^{(2)}+1,\ldots,L^{(1)}+L^{(2)}-1.
\end{align*}
\end{definition}
It follows immediately from this definition that
\[
\realiz{ \Phi^1 \bullet \Phi^2 } = \realiz{ \Phi^1 } \circ \realiz{ \Phi^2 }.
\]

\subsection{NN Emulations of Lowest Order Conforming FE Spaces.}
\label{sec:approximation}

In order to state the LSQ-G problem in terms of NNs,
we recall from \cite{LODSZ22_991}
\emph{exact NN emulations} 
of each of the lowest order FE spaces in the de Rham complex,
on regular,
simplical partitions $\calT$ of polytopal domains $\domain \subset \bbR^d$.


For $\blacklozenge \in \{ \So, \Ne, \RT, \Sz \}$,
we have by \cite[Section 5]{LODSZ22_991} a vector space of NNs
\[
\calNN(\blacklozenge;\calT,\domain) 
	= \{ \Phi^{\blacklozenge,v} : v\in\blacklozenge(\calT,\domain) \}
\]
such that 
%
\[
\forall v \in \blacklozenge(\calT,\domain): \quad 
\realiz{\Phi^{\blacklozenge,v}} = v \quad\mbox{a.e. in} \;\; \domain\;.
\]
\begin{proposition}[{{\cite[Proposition 5.1]{LODSZ22_991}}}]
\label{prop:basisnet}
Let $\domain\subset\R^d$, $d\geq2$, be a bounded, polytopal domain.
For every regular, simplicial triangulation $\calT$ of $\domain$
and every $\blacklozenge \in \{ \So, \Ne, \RT, \Sz \}$
(with the N\'{e}d\'{e}lec space $\blacklozenge = \Ne$ excluded if $d > 3$),
there exists a NN $\Phi^\blacklozenge :=\Phi^{\blacklozenge(\calT,\domain)}$
with ReLU and BiSU activations,
which in parallel emulates the shape functions $\{ \theta^\blacklozenge_i \}_{i\in \calI}$
for $\calI \in \{ \calV, \calE, \calF, \calT \}$, respectively,
that is $ \realiz{\Phi^\blacklozenge}\colon \domain \to \R^{|\calI|} $ satisfies
\begin{equation*}
    \realiz{\Phi^\blacklozenge}_i(x)
    = \, \theta^\blacklozenge_i(x)
    \quad\text{ for a.e. } x\in\domain
    \text{ and all } i\in\calI.
\end{equation*}
There exists $C>0$ independent of $d$ and $\calT$ such that
\begin{align*}
\depth(\Phi^\blacklozenge)
	= &\, \begin{cases} 5 & \text{ if } \blacklozenge \in \{ \So, \Ne, \RT \},
		\\ 3 &  \text{ if } \blacklozenge = \Sz, \end{cases}
	\\
    \size(\Phi^\blacklozenge) \le
  & C d^2 \sum_{i\in \calI} s(i) \le C d^2\mathfrak{s}(\calI) \dim( \blacklozenge(\calT,\domain)).
\end{align*}

For $\blacklozenge \in \{ \So, \Ne, \RT, \Sz \}$
and for every FE function
$v = \sum_{i\in\calI} v_i \theta^\blacklozenge_i \in\blacklozenge(\calT,\domain)$,
there exists a NN
$\Phi^{\blacklozenge,v} := \Phi^{\blacklozenge(\calT,\domain),v}$
with ReLU and BiSU activations,
such that for a constant $C>0$ independent of $d$ and $\calT$
\begin{align*}
\realiz{\Phi^{\blacklozenge,v}}(x)
	= &\, v(x)
	\quad\text{ for a.e. } x\in\domain
	,
	\\
\depth(\Phi^{\blacklozenge,v})
	= &\, \begin{cases} 5 & \text{ if } \blacklozenge \in \{ \So, \Ne, \RT \},
		\\ 3 &  \text{ if } \blacklozenge = \Sz, \end{cases}
	\\
    \size(\Phi^{\blacklozenge,v})
    \leq
    &\, C d^2 \sum_{i\in\calI} s(i)
     \leq C d^2 \mathfrak{s}(\calI) \dim( \blacklozenge(\calT,\domain)).
\end{align*}

The layer dimensions and the lists of activation functions
of $\Phi^\blacklozenge$ and $\Phi^{\blacklozenge,v}$
are independent of $v$
and only depend on $\calT$ through
    $ \{s(i)\}_{i\in\calI} $ and $ \snorm{\calI} = \dim(\blacklozenge(\calT,\domain)) $.

For each $\blacklozenge \in \{ \So, \Ne, \RT, \Sz \}$,
the set 
\begin{equation}\label{eq:NNFEMDef}
\calNN(\blacklozenge;\calT,\domain)
	:= \{ \Phi^{\blacklozenge,v} : v \in \blacklozenge(\calT,\domain) \}\;,
\end{equation}
together with the linear operation
\begin{align}
\label{eq:NNFEMlin}
\Phi^{\blacklozenge,v} \widehat{+} \lambda\Phi^{\blacklozenge,w} 
	:= &\, \Phi^{\blacklozenge,v+\lambda w},
	\qquad
	\text{ for all } v,w\in \blacklozenge(\calT,\domain) \text{ and }\lambda\in\R
\end{align}
is a vector space,
and ${\rm R} : \calNN(\blacklozenge;\calT,\domain) \to \blacklozenge(\calT,\domain)$
a linear bijection.
\end{proposition}
For a proof, we refer to \cite[Section 5]{LODSZ22_991}.

\begin{remark}[{{\cite[Remark 5.2]{LODSZ22_991}}}]
Note that $\sum_{i \in \calI} s(i) \leq c(\calI,d) \snorm\calT$,
where $c(\calV,d) = d+1$ is the number of vertices of a
$d$-simplex, $c(\calE,d)$ the number of edges of a $d$-simplex,
$c(\calF,d)$ the number of faces of a $d$-simplex and
$C(\calT,d) = 1$.  We obtain this inequality by observing that
each element $K \in \calT$ contributes $+1$ to $c(\calI,d)$
terms $s(i)$.  Therefore, we also have the bound
$ \size(\Phi^{\blacklozenge}) \leq Cd^2 c(\calI,d)
\snorm\calT$.
The same bound holds for $\size(\Phi^{\blacklozenge,v})$.
\end{remark}

\begin{definition}[{{\cite[Definition 5.3]{LODSZ22_991}}}]
\label{def:basisnet}
For given polytopal $ \domain \subset \R^d $, $ d\ge2 $
and a regular, simplicial triangulation $ \calT $ on $ \domain $,
we call the NN $ \Phi^{\blacklozenge} $
defined in Proposition \ref{prop:basisnet} a \emph{$\blacklozenge$-basis NN}.
\end{definition}

The following analogue of Proposition \ref{prop:basisnet} for ReLU emulation of $\So$ also holds.
Here, the dimension $d\geq 2$ of the physical domain $\domain$ is arbitrary, 
and the NN size parameters are explicit in $d$.

\begin{proposition}[{{\cite[Proposition 5.7]{LODSZ22_991}}}]
  \label{prop:CPLbasisnet}
  Let $\domain\subset\R^d$, $d\geq2$, be a bounded, polytopal domain.
  For every regular, simplicial triangulation $\calT$ of $\domain$,
  there exists a NN $\Phi^{CPwL} :=\Phi^{CPwL(\calT,\domain)}$ with
  only ReLU activations, which in parallel emulates the shape
  functions $\{ \theta^{\So}_i \}_{i\in \calI}$ for $\calI = \calV$.
  That is, $ \realiz{\Phi^{CPwL}}\colon \domain \to \R^{|\calI|} $
  satisfies
  \begin{equation*}
    \realiz{\Phi^{CPwL}}_i(x)
    = \, \theta^{\So}_i(x)
    \quad\text{ for all } x\in\domain
    \text{ and all } i\in\calI.
  \end{equation*}
  There exists $C>0$ independent of $d$ and $\calT$ such that
  \begin{align*}
    \depth(\Phi^{CPwL})
    \leq &\, 8 + \log_2( \mathfrak{s}(\calI) ) + \log_2(d+1)
           ,
    \\
    \size(\Phi^{CPwL})
    \leq & C \snorm{\calI} \log_2( \mathfrak{s}(\calI) ) + C d^2 \sum_{i\in\calI} s(i)
           \leq C d^2 \mathfrak{s}(\calI) \dim( \So(\calT,\domain)).
  \end{align*}

  For all
  $v = \sum_{i\in\calI} v_i \theta^{\So}_i \in\So(\calT,\domain)$,
  there exists a NN $\Phi^{CPwL,v} := \Phi^{CPwL(\calT,\domain),v} $
  with only ReLU activations, such that for a constant $C>0$
  independent of $d$ and $\calT$
  \begin{align*}
    \realiz{\Phi^{CPwL,v}}(x)
    = &\, v(x)
	\quad\text{ for all } x\in\domain
	,
    \\
    \depth(\Phi^{CPwL,v})
    \leq &\, 8 + \log_2( \mathfrak{s}(\calI) ) + \log_2(d+1)
           ,
    \\
    \size(\Phi^{CPwL,v})
    \leq &\, C \snorm{\calI} \log_2( \mathfrak{s}(\calI) ) + C d^2 \sum_{i\in\calI} s(i)
           \leq C d^2 \mathfrak{s}(\calI) \dim( \So(\calT,\domain)).
  \end{align*}

  The layer dimensions and the lists of activation functions of
  $\Phi^{CPwL}$ and $\Phi^{CPwL,v}$ are independent of $v$ and only
  depend on $\calT$ through $ \{s(i)\}_{i\in\calI} $ and
  $ \snorm{\calI} = \dim(\So(\calT,\domain)) $.

  The set
  $\calNN(CPwL;\calT,\domain) := \{ \Phi^{CPwL,v} : v \in
  \So(\calT,\domain) \}$ together with the linear operation
  $\Phi^{CPwL,v} \widehat{+} \lambda\Phi^{CPwL,w} := \Phi^{CPwL,v+\lambda w}$
  for all $v,w\in \So(\calT,\domain)$ and all $\lambda\in\R$ is a
  vector space.
  The realization map
  ${\rm R} : \calNN(CPwL;\calT,\domain) \to \So(\calT,\domain)$
  is a linear bijection.
\end{proposition}

\subsection{Deep FoSLS}
\label{sec:deepFoSLS}
By combining multiple NNs from Section \ref{sec:approximation},
we obtain emulations of finite-dimensional discretization
spaces $\bbV_\bullet(\domain)$ as in Section \ref{sec:LSQFEM}.

We slightly generalize the notation from Section \ref{sec:approximation}
to provide NN emulations of these spaces. 
In particular, we need
Cartesian products of the spaces treated in Section \ref{sec:approximation},
with some of the spaces in these cartesian product 
carrying homogeneous essential boundary conditions.
We use the general notation 
$\blacklozenge := \blacklozenge(\calT,\domain)$
and assume a Cartesian product structure
\[
\blacklozenge = \blacklozenge_1 \times \ldots \times \blacklozenge_m
\]
for some $m\in\N$ 
(with $m=1$ corresponding to a single space $\blacklozenge(\calT,\domain)$).
For each factor space 
\[
\blacklozenge_k := \blacklozenge_k(\calT,\domain) \;\; \mbox{for} \;\; k=1,\ldots,m
\]
(we assume one common triangulation $\calT$ of $\domain$ for all $\blacklozenge_k$)
we consider a separate basis of \emph{shape function features}\footnote{
See Remark \ref{rem:foslsbasisnetreuse} for an alternative construction.}
$\{ \theta^{\blacklozenge_k}_j \}_{j\in\calI_k}$
for $\calI_k \in \{ \calV,\calE,\calF,\calT \}$,
or for $\calI_k$ a subset of $\calV,\calE,\calF$ or $\calT$
in case homogeneous boundary conditions are imposed on $\blacklozenge_k$.
Each shape function of $\blacklozenge(\calT,\domain)$
is associated to one of the $m$ factor spaces.
To keep track of this,
we define $\calI := \cup_{k=1}^m \calI_k \times \{ k \}$,
whose elements are of the form $\calI \ni i = (j,k)$,
for a simplex $j \in \cup_{k=1}^m \calI_k \subset \calV\cup\calE\cup\calF\cup\calT$
and $k\in\N$.
With the definition
$\theta^{\blacklozenge}_i := \theta^{\blacklozenge_k}_j$,
for all $v\in \blacklozenge$ 
there exist coefficients $(v_i)_{i\in\calI} \subset\R$ such that
for all $k=1,\ldots,m$ the $k$'th component of $v$, 
which we denote by $v_{(k)}$, 
can be expanded as
$v_{(k)} = \sum_{ j\in\calI_k } v_{(j,k)} \theta^{\blacklozenge_k}_j
	= \sum_{i = (j,k)\in\calI} v_i \theta^{\blacklozenge}_i$,
where the sum is over all indices $i \in \calI$ whose second component equals $k$.
This generalizes the expansion 
$v = \sum_{i\in\calI} v_i \theta^{\blacklozenge}_i$
used in Section \ref{sec:approximation}
for spaces without Cartesian product structure.
For consistency of notation,
for all $i = (j,k) \in \calI$ we denote by
$s(i) := s(j) = \# \{ K \in \calT : j \subset K \}$
the number of elements of $\calT$ sharing the subsimplex $j$,
and define
$\mathfrak{s}(\calI) := \max_{i\in\calI} s(i)$.

For example, consider the spaces $\bbV_\bullet(\domain)$
defined in \eqref{eq:bbVPoisson}--\eqref{eq:bbVParab},
which we will denote by
$\bbV_\bullet^{(1)}(\domain)$, 
$\bbV_\bullet^{(2)}(\domain)$, 
$\bbV_\bullet^{(3)}(\domain)$ and
$\bbV_\bullet^{(4)}(\domain)$,
respectively.
We sometimes denote the mesh explicitly
and write
$\blacklozenge 
	= \blacklozenge(\calT,\domain) 
	= \bbV_\bullet^{(i)}(\calT,\domain)$ 
for $i\in\{1,2,3,4\}$.
We will denote the corresponding index sets by 
$\calI^{(i)}$.
For $i=1$, we have $m=2$, 
$\blacklozenge_1 = \So \cap H^1_0(\domain)$,
$\calI_1 = \calV \cap \domain$,
$\blacklozenge_2 = \RT$ and
$\calI_2 = \calF$.
Here, we impose homogeneous Dirichlet boundary conditions on $\blacklozenge_1$
by restricting $\calI_1$ to interior vertices.
For $i=2$, we have $m=2d$, 
$\blacklozenge_1 = \ldots = \blacklozenge_d = \So \cap H^1_0(\domain)$,
$\calI_1 = \ldots = \calI_d = \calV \cap \domain$,
$\blacklozenge_{d+1} = \ldots = \blacklozenge_{2d} = \RT$ 
and
$\calI_{d+1} = \ldots = \calI_{2d} = \calF$.
For $i=3$, we have $m=2$, 
$\blacklozenge_1 = \Ne \cap H_0(\curl;\domain)$,
$\calI_1 = \{ e\in \calE : e \cap D \neq \emptyset \}$,
$\blacklozenge_2 = \Ne$ and
$\calI_2 = \calE$.
For $i=4$, we have $m=d+1$, 
$\blacklozenge_1 = \So \cap L^2(I;H^1_0(\spatialdomain))$,
$\calI_1 = \calV \cap ( \overline{I} \times \spatialdomain )$,
$\blacklozenge_2 = \ldots = \blacklozenge_{d+1} = \So$ and
$\calI_2 = \ldots = \calI_{d+1} = \calV$.

\begin{proposition}
\label{prop:foslsbasisnet}
Let $\domain\subset\R^d$, $d\geq2$, be a bounded, polytopal domain.
For every regular, simplicial triangulation $\calT$ of $\domain$
and every 
$\blacklozenge \in \{ \bbV_\bullet^{(1)}, 
	\bbV_\bullet^{(2)}, \bbV_\bullet^{(3)}, \bbV_\bullet^{(4)} \}$
(with the N\'{e}d\'{e}lec space $\bbV_\bullet^{(3)}$ excluded if $d > 3$),
there exists a NN $\Phi^\blacklozenge := \Phi^{\blacklozenge(\calT,\domain)}$
with ReLU and BiSU activations,
which in parallel emulates the shape functions $\{ \theta^\blacklozenge_i \}_{i\in \calI}$
for $\calI \in \{ \calI^{(1)}, \calI^{(2)}, \calI^{(3)}, \calI^{(4)} \}$, 
respectively,
that is $ \realiz{\Phi^\blacklozenge}\colon \domain \to \R^{|\calI|} $ 
satisfies
\begin{equation*}
    \realiz{\Phi^\blacklozenge}(x)_i
    = \, \theta^\blacklozenge_i(x)
    \quad\text{ for a.e. } x\in\domain
    \text{ and all } i\in\calI.
\end{equation*}
There exists $C>0$ independent of $d$ and $\calT$ such that
\begin{align*}
\depth(\Phi^\blacklozenge)
	= &\, 5
	,
	\qquad
    \size(\Phi^\blacklozenge) \le
	C d^2 \sum_{i\in \calI} s(i) \le C d^2\mathfrak{s}(\calI) \dim( \blacklozenge(\calT,\domain)).
\end{align*}

For $\blacklozenge \in \{ \bbV_\bullet^{(1)}, 
	\bbV_\bullet^{(2)}, \bbV_\bullet^{(3)} \}$
and for every FE function satisfying
$v_{(k)} = \sum_{i = (j,k)\in\calI} v_i \theta^{\blacklozenge}_i \in\blacklozenge_k(\calT,\domain)$
for all $k=1,\ldots,m$,
there exists a NN
$\Phi^{\blacklozenge,v} := \Phi^{\blacklozenge(\calT,\domain),v}$
with ReLU and BiSU activations,
such that for a constant $C>0$ independent of $d$ and $\calT$
\begin{align*}
\realiz{\Phi^{\blacklozenge,v}}_k(x)
	= &\, v_{(k)}(x)
	\quad\text{ for a.e. } x\in\domain
	\text{ and all } k=1,\ldots,m
	,
	\\
\depth(\Phi^{\blacklozenge,v})
	= &\, 5
	,
	\qquad
    \size(\Phi^{\blacklozenge,v})
    \leq
    C d^2 \sum_{i\in\calI} s(i)
     \leq C d^2 \mathfrak{s}(\calI) \dim( \blacklozenge(\calT,\domain)).
\end{align*}

The layer dimensions and the lists of activation functions
of $\Phi^\blacklozenge$ and $\Phi^{\blacklozenge,v}$
are independent of $v$
and only depend on $\calT$ through
    $ \{s(i)\}_{i\in\calI} $ and $ \snorm{\calI} = \dim(\blacklozenge(\calT,\domain)) $.

For each $\blacklozenge \in \{ \bbV_\bullet^{(1)}, 
	\bbV_\bullet^{(2)}, \bbV_\bullet^{(3)} \}$,
the set 
\begin{equation}\label{eq:foslsNNFEMDef}
\calNN(\blacklozenge;\calT,\domain)
	:= \{ \Phi^{\blacklozenge,v} : v \in \blacklozenge(\calT,\domain) \}\;,
\end{equation}
together with the linear operation
\begin{align}
\label{eq:foslsNNFEMlin}
\Phi^{\blacklozenge,v} \widehat{+} \lambda\Phi^{\blacklozenge,w} 
	:= &\, \Phi^{\blacklozenge,v+\lambda w},
	\qquad
	\text{ for all } v,w\in \blacklozenge(\calT,\domain) \text{ and }\lambda\in\R
\end{align}
is a vector space,
and ${\rm R} : \calNN(\blacklozenge;\calT,\domain) \to \blacklozenge(\calT,\domain)$
a linear isomorphism.
\end{proposition}

\begin{proof}
We use the same arguments as in the proof of \cite[Proposition 5.1]{LODSZ22_991}.

For $k=1,\ldots,m$
let $\Phi^{\blacklozenge_k(\calT,\domain)}$ 
be constructed as in the proof of \cite[Proposition 5.1]{LODSZ22_991}
(which is Proposition \ref{prop:basisnet})
as the parallelization 
$\Phi^{\blacklozenge_k(\calT,\domain)} 
	:= \Parallel{ \{ \Phi^{\blacklozenge_k}_j \}_{j\in\calI_k} }$
of subnetworks $\Phi^{\blacklozenge_k}_j$ 
which exactly emulate the shape functions $\theta^{\blacklozenge_k}_j$.
We deviate slightly from the construction in \cite{LODSZ22_991}.
By definition of $\calI_k$, 
which can be a strict subset of $\calV,\calE,\calF$ or $\calT$,
we omit shape functions which are associated to elements $j$ 
on the part of the boundary where homogeneous boundary conditions are imposed.
As a result, we have
$\snorm{\calI_k} = \dim \blacklozenge_k(\calT,\domain)$ 
also for those spaces.

Then, we define 
$\Phi^{\blacklozenge(\calT,\domain)} 
	:= \Parallel{ \{ \Phi^{\blacklozenge_k(\calT,\domain)} \}_{k=1}^m }$,
from which the formula for the realization follows.
All components have depth $5$, so the same holds for the parallelization.
The NN size satisfies 
\begin{align*}
\size(\Phi^\blacklozenge) 
	\leq &\, \sum_{k=1}^m \size(\Phi^{\blacklozenge_k}) 
	\leq \sum_{k=1}^m C d^2 \sum_{j\in \calI_k} s(j)
	= C d^2 \sum_{i\in\calI} s(i)
	\leq C d^2\mathfrak{s}(\calI) \dim( \blacklozenge(\calT,\domain)).
\end{align*}

Similarly, for all 
$v \in \blacklozenge(\calT,\domain)$
satisfying
$v_{(k)} = \sum_{ j\in\calI_k } v_{(j,k)} \theta^{\blacklozenge_k}_j$ 
for all $k=1,\ldots,m$,
we consider
$\Phi^{\blacklozenge_k(\calT,\domain),v_{(k)}}
	:= \sum_{j\in\calI_k} v_{(j,k)} \Phi^{\blacklozenge_k}_j$,
as constructed in the proof of \cite[Proposition 5.1]{LODSZ22_991},
where the sum of NNs is according to Proposition \ref{prop:sum}
and where we implement the scalar multiplication with $v_{(j,k)}$
by scaling the weights in the output layer of $\Phi^{\blacklozenge_k}_j$.
Again, we omitted shape functions associated to elements
on the part of the boundary where homogeneous boundary conditions are imposed.

We define 
$\Phi^{\blacklozenge(\calT,\domain),v} 
	:= \Parallel{ \{ \Phi^{\blacklozenge_k(\calT,\domain),v_{(k)}} \}_{k=1}^m }$,
from which the formula for the realization follows.
The NN has depth $5$
and the bound on the NN size is analogous to that of
$\Phi^{\blacklozenge(\calT,\domain)}$.

By comparing Proposition \ref{prop:parallel} and \ref{prop:sum},
we observe that the parallelization 
$\Phi^{\blacklozenge_k(\calT,\domain)}$ 
and NN sum 
$\Phi^{\blacklozenge_k(\calT,\domain),v_{(k)}}$
have identical hidden layers,
which are independent of $v_{(k)}$ 
and only depend on $\calT$ through $ \{s(j)\}_{j\in\calI_k} $ and $ \snorm{\calI_k} $.
Also their parallelizations
$\Phi^{\blacklozenge(\calT,\domain)}$ 
and 
$\Phi^{\blacklozenge(\calT,\domain),v}$
have identical hidden layers,
which are independent of $v$
and only depend on $\calT$ through $ \{s(i)\}_{i\in\calI} $ and $\snorm{\calI} $.

The linear structure of $\calNN(\blacklozenge;\calT,\domain)$
follows from the linear structure of the NNs in Proposition \ref{prop:basisnet}.
\end{proof}

\begin{definition}
\label{def:foslsbasisnet}
For given polytopal $ \domain \subset \R^d $, $ d\ge2 $
and a regular, simplicial triangulation $ \calT $ on $ \domain $,
we call the NN $ \Phi^{\blacklozenge} $
defined in Proposition \ref{prop:foslsbasisnet} a \emph{$\blacklozenge$-FoSLS basis NN}.
\end{definition}

\begin{remark}
\label{rem:foslsbasisnet}
For all $v \in\blacklozenge(\calT,\domain)$,
denoting by 
$\bsv_k = ( v_{(j,k)} )_{j\in\calI_k} 
	\in \R^{1\times \snorm{\calI_k}}$
the row vector of coefficients for the $k$'th component of $v$
with respect to the shape functions $\{ \theta^{\blacklozenge_k}_j \}_{j\in\calI_k}$,
the NN $\Phi^{\blacklozenge,v}$ can be obtained from $\Phi^\blacklozenge$
as follows.
Denoting the last layer weight matrix and bias vector of $\Phi^{\blacklozenge_k}$ by
$A^{(k)}$ and $b^{(k)}$,
those of $\Phi^\blacklozenge$ are given by
$A = \operatorname{diag}( A^{(1)}, \ldots, A^{(m)} )$ and
$b = ( ( b^{(1)} )^\top, \ldots, ( b^{(m)} )^\top )^\top$,
and those of $\Phi^{\blacklozenge,v}$ are given by
$\operatorname{diag}( \bsv_1 A^{(1)}, \ldots, \bsv_m A^{(m)} )$ 
and 
$( \bsv_1 b^{(1)} , \ldots, \bsv_m b^{(m)} )^\top$.
\end{remark}

\begin{remark}
\label{rem:foslsbasisnetreuse}
When multiple factors from $\blacklozenge_1,\ldots,\blacklozenge_m$
coincide (or are subspaces of each other, 
in case of homogeneous boundary conditions),
then a more efficient construction is possible
in which the shape functions are emulated only once,
but used for the emulation of multiple factor spaces.

We will demonstrate this for the spaces 
$\bbV_\bullet^{(3)}$ and $\bbV_\bullet^{(4)}$,
where all the factor spaces are based on the same finite element.
The same idea can be used if some but not all of the factor spaces
are based on the same finite element.

We emulate each shape function only once.
%
If $\blacklozenge = \bbV_\bullet^{(3)}$
we define 
$\blacklozenge_* = \Ne$ and 
$\calI_* = \calE$,
and if $\blacklozenge = \bbV_\bullet^{(4)}$
we define
$\blacklozenge_* = \So$ and 
$\calI_* = \calV$.
Then, let 
$\Phi^{\blacklozenge_*(\calT,\domain)}$ 
be as given by Proposition \ref{prop:basisnet}.
We obtain that
\begin{equation*}
    \realiz{\Phi^{\blacklozenge_*}}(x)_i
    = \, \theta^{\blacklozenge_*}_i(x)
    \quad\text{ for a.e. } x\in\domain
    \text{ and all } i\in\calI_*,
\end{equation*}
and that
there exists $C>0$ independent of $d$ and $\calT$ such that
\begin{align*}
\depth(\Phi^{\blacklozenge_*})
	= &\, 5
	,
	\qquad
\size(\Phi^{\blacklozenge_*})
    \leq C d^2 \sum_{i\in \calI_*} s(i) 
    \leq C d^2\mathfrak{s}(\calI_*) \dim( \blacklozenge_*(\calT,\domain))
    .
\end{align*}
The gain is that $\calI_* \in \{ \calE, \calV \}$ 
is much smaller than $\calI$ in Proposition \ref{prop:foslsbasisnet}.

For all $v \in\blacklozenge(\calT,\domain)$,
denoting by 
$\bsv_k = ( v_{(j,k)} )_{j\in\calI_*} 
	\in \R^{1\times \snorm{\calI_*}}$
the row vector of coefficients for the $k$'th component of $v$,
defining $v_{(j,k)} = 0$ if $j\in\calI_*\setminus\calI_k$ 
(which may be the case when homogeneous boundary conditions are imposed 
on $\blacklozenge_k$),
we now define $\widetilde\Phi^{\blacklozenge,v}$ as follows.
Denoting the last layer weight matrix and bias vector of
$\Phi^{\blacklozenge_*}$ by $A$ and $b$,
we define those of $\widetilde\Phi^{\blacklozenge,v}$ to be
$( ( \bsv_1 A )^\top, \ldots, ( \bsv_m A )^\top )^\top$ 
and 
$( \bsv_1 b , \ldots, \bsv_m b )^\top$.
For a constant $C>0$ independent of $d$ and $\calT$ it holds that
\begin{align*}
\realiz{\widetilde\Phi^{\blacklozenge,v}}_k(x)
	= &\, v_{(k)}(x)
	\quad\text{ for a.e. } x\in\domain
	\text{ and all } k=1,\ldots,m
	,
	\\
\depth(\Phi^{\blacklozenge,v})
	= &\, 5
	,
	\qquad
\size(\widetilde\Phi^{\blacklozenge,v})
    \leq
    C d^2 \sum_{i\in\calI} s(i)
     \leq C d^2 \mathfrak{s}(\calI) \dim( \blacklozenge(\calT,\domain)).
\end{align*}
Because the shape functions are emulated only once,
the hidden layers of 
$\widetilde\Phi^{\blacklozenge,v}$ 
have smaller size than those of 
$\Phi^{\blacklozenge,v}$
(compare the bounds on 
$\size(\Phi^{\blacklozenge})$ and $\size(\Phi^{\blacklozenge_*})$).
For the number of nonzero weights and biases in the output layer
of $\widetilde\Phi^{\blacklozenge,v}$ 
we do not obtain a better bound
than for that of $\Phi^{\blacklozenge,v}$,
hence the reduction in NN size is not visible 
in the bound on $\size(\widetilde\Phi^{\blacklozenge,v})$ stated above.
\end{remark}

The following analogue of Proposition \ref{prop:foslsbasisnet} 
for ReLU emulation of $\bbV_\bullet^{(4)}$ also holds.
\begin{proposition}
\label{prop:foslsCPLbasisnet}
Let $\domain\subset\R^d$, $d\geq2$, be a bounded, polytopal domain
and let $\blacklozenge = \bbV_\bullet^{(4)}$.
For every regular, simplicial triangulation $\calT$ of $\domain$,
there exists a strict ReLU NN 
$\Phi^{CPwL\;\blacklozenge} :=\Phi^{CPwL\;\blacklozenge(\calT,\domain)}$ 
which emulation (in parallel) the shape functions 
$\{ \theta^{\blacklozenge}_i \}_{i\in \calI}$.
That is, 
$\realiz{\Phi^{CPwL\;\blacklozenge}}\colon \domain \to \R^{|\calI|} $
satisfies
\begin{equation*}
\realiz{\Phi^{CPwL\;\blacklozenge}}_i(x)
= \, \theta^{\blacklozenge}_i(x)
\quad\text{ for all } x\in\domain
\text{ and all } i\in\calI.
\end{equation*}
There exists $C>0$ independent of $d$ and $\calT$ such that
\begin{align*}
\depth(\Phi^{CPwL\;\blacklozenge})
\leq &\, 8 + \log_2( \mathfrak{s}(\calI) ) + \log_2(d+1)
,
\\
\size(\Phi^{CPwL\;\blacklozenge})
\leq & C \snorm{\calI} \log_2( \mathfrak{s}(\calI) ) + C d^2 \sum_{i\in\calI} s(i)
\leq C d^2 \mathfrak{s}(\calI) \dim( \blacklozenge(\calT,\domain)).
\end{align*}

For all
$v \in\blacklozenge(\calT,\domain)$,
there exists a NN $\Phi^{CPwL\;\blacklozenge,v} := \Phi^{CPwL\;\blacklozenge(\calT,\domain),v} $
with only ReLU activations, such that for a constant $C>0$
independent of $d$ and $\calT$
\begin{align*}
\realiz{\Phi^{CPwL\;\blacklozenge,v}}_k(x)
= &\, v_{(k)}(x)
\quad\text{ for all } x\in\domain\text{ and all }k=1,\ldots,m
,
\\
\depth(\Phi^{CPwL\;\blacklozenge,v})
\leq &\, 8 + \log_2( \mathfrak{s}(\calI) ) + \log_2(d+1)
,
\\
\size(\Phi^{CPwL\;\blacklozenge,v})
\leq &\, C \snorm{\calI} \log_2( \mathfrak{s}(\calI) ) + C d^2 \sum_{i\in\calI} s(i)
\leq C d^2 \mathfrak{s}(\calI) \dim( \blacklozenge(\calT,\domain)).
\end{align*}

The layer dimensions and the lists of activation functions of
$\Phi^{CPwL\;\blacklozenge}$ and $\Phi^{CPwL\;\blacklozenge,v}$ are independent of $v$ and only
depend on $\calT$ through $ \{s(i)\}_{i\in\calI} $ and
$ \snorm{\calI} = \dim(\blacklozenge(\calT,\domain)) $.

The set
$\calNN(CPwL\;\blacklozenge;\calT,\domain) := \{ \Phi^{CPwL\;\blacklozenge,v} : v \in
\blacklozenge(\calT,\domain) \}$ together with the linear operation
$\Phi^{CPwL\;\blacklozenge,v} \widehat{+} \lambda\Phi^{CPwL\;\blacklozenge,w} := \Phi^{CPwL\;\blacklozenge,v+\lambda w}$
for all $v,w\in \blacklozenge(\calT,\domain)$ and all $\lambda\in\R$ is a
vector space.
The realization map
${\rm R} : \calNN(CPwL\;\blacklozenge;\calT,\domain) \to \blacklozenge(\calT,\domain)$
is a linear bijection.
\end{proposition}

\begin{proof}
The proof is analogous to that of Proposition \ref{prop:foslsbasisnet},
except for the fact that 
the depth of the subnetworks corresponding to the different factor spaces
may be larger than $5$.
In fact, when on one of the factor spaces homogeneous boundary conditions are imposed,
the depth of the corresponding subnetwork 
may be smaller than the depth of some of the other subnetworks.
If that is the case, 
we concatenate the subnetworks which are not of maximal depth
with a ReLU identity NN,
precisely as in the proof of \cite[Proposition 5.7]{LODSZ22_991}.
\end{proof}

The alternative approach from Remark \ref{rem:foslsbasisnetreuse}
can also be applied to the deep ReLU NNs from Proposition \ref{prop:CPLbasisnet},
which gives ReLU NNs that emulate the same functions
as those in Proposition \ref{prop:foslsCPLbasisnet},
but have a smaller NN size.

\begin{remark}
\label{rem:deepfoslsocp}
We stated all results in this section for 
$\blacklozenge \in \{ \bbV_\bullet^{(1)}, 
	\bbV_\bullet^{(2)}, \bbV_\bullet^{(3)}, \bbV_\bullet^{(4)} \}$,
but they hold for arbitrary Cartesian products
of spaces emulated in Section \ref{sec:approximation},
with homogeneous boundary conditions imposed where necessary.

In particular, they also hold for the discretizations of the optimal control problems.
For those, we need to discretize $\bbV_Y \times \bbV_P$,
where each of the factors $\bbV_Y, \bbV_P$ 
is discretized by $\bbV_\bullet^{(\ell)}(\domain)$ for $\ell\in\{1,3,4\}$.
Thus, we have $m=4$ if $\ell\in\{1,3\}$ and $m=2d+2$ if $\ell=4$.
\end{remark}

From Section \ref{sec:LSQFEM}, 
we obtain that the FoSLS NNs from Proposition \ref{prop:foslsbasisnet} 
are quasioptimal.\footnote{The 
following analysis also holds for the NNs 
from Remark \ref{rem:foslsbasisnetreuse} 
and Proposition \ref{prop:foslsCPLbasisnet}.}
Indeed, setting
\begin{equation}\label{eq:deepFoSLSMinN}
\Phi_{\bullet} 
= 
{\rm arg}\min_{\widetilde\Phi_{\bullet} \in \calNN(\blacklozenge;\calT,\domain)} 
	\| F - \VL \realiz{\widetilde\Phi_{\bullet}} \|^2_{\bbL(\domain)}
\;,
\end{equation}
we observe that the argmin is well-defined and unique 
because $\calNN(\blacklozenge;\calT,\domain)$ is a finite dimensional space 
by Proposition \ref{prop:foslsbasisnet} and $\VL$ and $\rm R$ are linear bijections.
This implies that
\begin{align}
\nonumber
{c_\VL} \| U - \realiz{\Phi_\bullet} \|_{\bbV(\domain)}
	\leq \| F - \VL \realiz{\Phi_\bullet} \|_{\bbL(\domain)}
	= &\, \min_{\widetilde\Phi_{\bullet} \in \calNN(\blacklozenge;\calT,\domain)} 
	\| F - \VL \realiz{\widetilde\Phi_{\bullet}} \|_{\bbL(\domain)}
	\\\label{eq:deepLSQQO}
	\leq &\, C_\VL \min_{\widetilde\Phi_{\bullet} \in \calNN(\blacklozenge;\calT,\domain)} 
	\| U - \realiz{\widetilde\Phi_{\bullet}} \|_{\bbV(\domain)}
	\\\nonumber
	= &\, C_\VL \min_{v_\bullet \in \bbV_\bullet(\domain)} \| U - v_\bullet \|_{\bbV(\domain)}
,
\end{align}
where the last equality follows from the fact that 
the finite element functions are realized exactly by the FoSLS NNs
and the fact that the realization
${\rm R}: \calNN(\blacklozenge;\calT,\domain) \to \blacklozenge(\calT,\domain)$
is a bijection.
 
We close the discussion by remarking that
NNs can improve upon a fixed $\bbV_\bullet(\domain)$
if we allow their hidden layers, which encode mesh connectivity and 
element shapes, to adapt to the target function $U$.
A strategy for 
\emph{adaptive growth of NNs with certain optimality properties}
based on adaptive least squares finite elements is 
described in Section \ref{sec:AFEM}.
\section{Adaptive NN growth strategies}
\label{sec:AFEM}
In the LSQ-G method, the value of the LSQ functional $\LS$ 
evaluated at the LSQ-G approximation $U_\bullet$ in \eqref{eq:FoSLSMinN}
is a reliable and efficient estimate of the LSQ-G error, 
due to \eqref{eq:LSQQO}.
For LSQ-G, the adaptive finite element method (AFEM) algorithm from \cite{Siebert2011},
as used in \cite{FuePraeLSQ,GaStLSQ}, implies, in particular,
a convergent iterative discretization algorithm with mathematically 
guaranteed convergence and, under assumptions, also with guaranteed optimality.

It follows the usual
\Solve-\Estimate-\Mark-\Refine~loop
and is stated below as Algorithm \ref{algo:AFEM}.
Under suitable assumptions, 
which we first state together with the definitions of 
\Solve, \Estimate, \Mark~and \Refine,
the algorithm converges, as shown in \cite[Section 2]{Siebert2011},
which we recall in Proposition \ref{prop:plainconv}.
The application of the proposition 
in the context of Sections \ref{sec:Poisson}--\ref{sec:nLSQParab}
is the topic of Remark \ref{rem:plainconvexamples}.
Optimal control problems are discussed in 
Remarks \ref{rem:plainconvOCP} and \ref{rmk:Proj}.

For a conforming simplicial triangulation $\calT_\bullet$ of $\domain$
comprising closed simplices 
and a set of 
marked elements $\calM_\bullet \subset \calT_\bullet$,
let
$\Refine(\calT_\bullet,\calM_\bullet)$ 
be the application of newest vertex bisection (NVB)
to the mesh $\calT_\bullet$, 
bisecting all marked elements $K \in \calM_\bullet$ at least once.
Given a conforming initial triangulation $\calT_0$,
we only consider meshes $\calT_\bullet$
which have been obtained from $\calT_0$ by applications of \Refine.
We denote the set of all meshes which can be obtained in this way by 
$\bbT(\calT_0)$.
These meshes are locally quasi-uniform,
they satisfy \cite[Equation (2.4)]{Siebert2011}.
%
%

As finite dimensional spaces taking the role of $\bbV_\bullet(\domain)$ in \eqref{eq:FoSLSMinN}
we use FEM spaces on collections of regular, simplicial triangulations 
$\calT_\bullet \in\bbT(\calT_0)$ 
that have the following properties.
\begin{itemize}
\item[(S1)] 
Conforming:
$\bbV_\bullet(\domain) \subset \bbV(\domain)$,
\item[(S2)]
Nested:
for all $\calT_\bullet\in\bbT(\calT_0)$
and all refinements $\calT_\circ \in \bbT(\calT_\bullet)$ 
holds 
$\bbV_\bullet(\domain) \subset \bbV_\circ(\domain)$,
where $\bbV_\bullet(\domain)$, $\bbV_\circ(\domain)$
denote the FEM spaces on $\calT_\bullet$, $\calT_\circ$ respectively,
\item[(S3)]
Local approximation property:
for a dense subspace ${\bbD^s}(\domain) \subset \bbV(\domain)$
of functions with higher regularity, endowed with the norm
$\norm[\bbD^s(\domain)]{\cdot}$ 
satisfying (A3),
the following holds.
There exist\footnote{
Here, $s$ quantifies the approximability of functions in $\bbD^s(\domain)$
by elements in $\bbV_\bullet(\domain)$.
This can often be related to the smoothness of elements in ${\bbD^s}(\domain)$.
}
$C>0$ and $s>0$
such that 
for all $\calT_\bullet\in\bbT(\calT_0)$
there exists a linear interpolation operator 
$\calA_\bullet : \bbD^s(\domain) \to \bbV_\bullet(\domain)$
satisfying 
\begin{align*}
\norm[\bbV(K)]{ v - \calA_\bullet v } 
	\leq C \snorm{K}^{-s/d} \norm[\bbD^s(K)]{ v },
	\qquad
	\text{ for all } v \in {\bbD^s}(K)
	\text{ and } K \in \calT_\bullet
	,
\end{align*}
where $\snorm{K}$ denotes the $d$-dimensional Lebesgue measure of $K$.
\end{itemize}
By $\Solve(\calT_\bullet,\bbV_\bullet(\domain)) \in \bbV_\bullet(\domain)$
we denote the exact solution $U_\bullet$ of \eqref{eq:FoSLSMinN}.

As an error estimator\footnote{This finite element terminology
should not be confused with the concept of an estimator in statistics.},
we use the least squares functional \eqref{eq:FoSLSMin}.
For all $\calT_\bullet \in \bbT(\calT_0)$ 
and $u_\bullet\in\bbV_\bullet(\domain)$
we denote 
$\eta_\bullet(u_\bullet, K) = \norm[\bbL(K)]{ F - \VL u_\bullet }$ 
for all $K\in\calT_\bullet$
and 
$\eta_\bullet(u_\bullet, \calU_\bullet) 
	= \left( \sum_{K\in\calU_\bullet} \eta_\bullet(u_\bullet, K)^2 \right)^{1/2}$
for all subsets $\calU_\bullet \subset \calT_\bullet$.
We define 
$\Estimate(\calT_\bullet,u_\bullet) 
	= \{ \eta_\bullet(u_\bullet, K) \}_{K\in\calT_\bullet}$.
In the analysis below,
we use the following mesh-related notation:
$\domain_\bullet(K) = \cup_{K'\in\calT_\bullet :\atop K\cap K' \neq\emptyset} K'$
denotes the patch around $K$.
We stipulate the following assumptions.
\begin{itemize}
\item[(E1)]
Upper bound:
there exists $C>0$ such that
for all $v \in \bbV(\domain)$
\begin{align*}
\sum_{K\in\calT_\bullet} ( F - \VL U_\bullet, \VL v )_{\bbL(K)}
	\leq C \sum_{K\in\calT_\bullet} \eta_\bullet( U_\bullet, K) \norm[\bbV(\domain_\bullet(K))]{ v }
,
\end{align*}
where $U_\bullet$ denotes the exact solution of \eqref{eq:FoSLSMinN}.
\item[(E2)]
Local stability:
there exist $D \in L^2(\domain)$ and $C>0$ such that
for all $K\in\calT_\bullet$:
\begin{align*}
\eta_\bullet(U_\bullet,K) 
	\leq C ( \norm[\bbV(\domain_\bullet(K))]{ U_\bullet } 
		+ \norm[\bbL(\domain_\bullet(K))]{ D } )
.
\end{align*}
\end{itemize}
In the proof of \cite[Theorem 2]{FuePraeLSQ},
Assumptions (E1)--(E2) are shown to follow from
\begin{itemize}
\item[(L)] 
Local boundedness:
there exists $C>0$ such that for all $\calT_\bullet \in \bbT(\calT_0)$ holds
\begin{align*}
\norm[\bbL(K)]{ \VL v } \leq C \norm[\bbV(\domain_\bullet(K))]{ v }
	,
	\qquad
	\text{ for all } v \in \bbV(\domain)
	\text{ and } K \in \calT_\bullet
	.
\end{align*}
\end{itemize}
As stated in \cite[Remark 3]{FuePraeLSQ},
local boundedness (L) also implies \emph{local efficiency}:
for all $\calT_\bullet \in \bbT(\calT_0)$
\begin{align*}
\eta_\bullet( v_\bullet, K ) 
	= \norm[\bbL(K)]{ \VL ( U - v_\bullet ) } 
	\leq C \norm[\bbV(\domain_\bullet(K))]{ U - v_\bullet }
	,
\qquad
\text{ for all } v_\bullet \in \bbV_\bullet(\domain)
\text{ and } K \in \calT_\bullet
.
\end{align*}

Several widely-used marking strategies are admissible here.
For 
$u_\bullet \in \bbV_\bullet(\domain)$ 
it is only required that
the set of marked elements 
$\calM_\bullet = \Mark( \calT_\bullet, \{ \eta_\bullet(u_\bullet, K) \}_{K\in\calT_\bullet} ) 
	\subset \calT_\bullet$
satisfies the following assumption.
\begin{itemize}
\item[(M)]
Marking assumption:
for a function $g: [0,\infty)\to[0,\infty)$ 
which is continuous in $0$ and satisfies $g(0)=0$,
it holds that
\begin{align*}
\max_{K \in \calT_\bullet \setminus \calM_\bullet} \eta_\bullet(u_\bullet, K)
	\leq g\Big( \max_{K \in \calM_\bullet} \eta_\bullet(u_\bullet, K) \Big)
.
\end{align*}
\end{itemize}
See \cite[Section 4.1]{Siebert2011}
for examples of marking strategies satisfying (M), 
which include
the maximum strategy,
and D\"orfler's strategy as formulated in \cite[Section 4.1, Item (c)]{Siebert2011}.

\begin{algorithm}
\caption{Adaptive Finite Element algorithm}
\label{algo:AFEM}
\begin{algorithmic}[1] 
\Input 
Initial partition $\calT_0$ of $\domain$
\Loopstart For all $\ell = 0, 1, \ldots$ repeat lines \ref{line:solve}--\ref{line:refine}:
\State 
\label{line:solve}
$\Solve(\calT_\ell,\bbV_\ell(\domain)) 
	\to U_\ell$
\State 
\label{line:estimate}
$\Estimate(\calT_\ell,U_\ell)
	\to \{ \eta_\ell(U_\ell, K) \}_{K\in\calT_\ell}$
\State 
\label{line:mark}
$\Mark( \calT_\ell, \{ \eta_\ell(U_\ell, K) \}_{K\in\calT_\ell} )
	\to \calM_\ell$
\State 
\label{line:refine}
$\Refine(\calT_\ell,\calM_\ell)
	\to \calT_{\ell+1}$
\Output 
For all $\ell\in\N_0$
the mesh $\calT_\ell$, 
the approximation $U_\ell \in \bbV_\ell(\domain)$
and the error indicator $\eta_\ell(U_\ell,\calT_\ell)$
\end{algorithmic}
\end{algorithm}
%

\begin{proposition}[{{\cite[Theorem 2.1]{Siebert2011}}}]
\label{prop:plainconv}
Assume that (A1)--(A4), (S1)--(S3) and (E1)--(E2) hold
and that a marking strategy satisfying (M) is used.

Then the iterates of Algorithm \ref{algo:AFEM} converge, i.e.
$\norm[\bbV(\domain)]{ U_\ell - U } \to 0$ for $\ell \to \infty$.
By \eqref{eq:PDEUniq}, also $\LS(U_\ell;F) \to 0$ for $\ell \to \infty$.
\end{proposition}

\begin{remark}
\label{rem:plainconvexamples}
For the examples from Sections \ref{sec:Poisson}--\ref{sec:nLSQParab},
with the finite-dimensional subspaces from \eqref{eq:bbVPoisson}--\eqref{eq:bbVParab},
Assumptions (A1)--(A4), (S1)--(S3) and (L) are proved in
\cite[Section 3.1]{FuePraeLSQ}, 
\cite[Section 3.3]{FuePraeLSQ}, 
\cite[Section 3.4]{FuePraeLSQ} 
and 
\cite[Proofs of Theorems 2.3 and 3.3]{GaStLSQ}, 
respectively.
\end{remark}

In \cite[Theorem 3.1]{CS2018}, 
it was shown for the Poisson, Helmholtz, linear elasticity and Maxwell LSQ problems 
from Sections \ref{sec:Poisson}, \ref{sec:Elasticity} and \ref{sec:nLSQCEM},
with their discretization from 
\eqref{eq:bbVPoisson}, \eqref{eq:bbVElasticity} and \eqref{eq:bbVMaxw},
that the LSQ functional is asymptotically exact.
I.e., 
for all $\epsilon>0$ there exists $\delta>0$
such that if $\bbV_\bullet(\domain)$ has meshsize at most $\delta$,
then the 
exact PDE solution $U\in \bbV(\domain)$ of \eqref{eq:PDE} 
and the 
approximate solution $U_\bullet\in \bbV_\bullet(\domain)$ 
which minimizes the LSQ-functional in \eqref{eq:FoSLSMinN}
satisfy
\begin{align*}
(1-\epsilon) \| U - U_\bullet \|_{\bbV(\domain)}
\leq
\| F - \VL U_\bullet \|_{\bbL(\domain)}
\leq 
(1+\epsilon) \| U - U_\bullet \|_{\bbV(\domain)}
.
\end{align*}

For the differential operator 
from Section \ref{sec:Poisson}, \ref{sec:nLSQCEM} or \ref{sec:nLSQParab},
we recalled in Sections \ref{sec:OCPPoisson}--\ref{sec:OCPSpacetime} 
the corresponding OCP
from \cite[Section 3]{TFMKOptCtrl2022}.
Examples of finite dimensional subspaces for their discretization 
were stated in Section \ref{sec:LSQFEM}.
We now recall in Remark \ref{rem:plainconvOCP} 
that for these discretizations the
hypotheses (S1)--(S3) and (L) are satisfied 
for the FoSLS formulation \eqref{eq:FoSLSOCP} of the OCP.
This means that it can be solved using Algorithm \ref{algo:AFEM},
as explained in Remark \ref{rmk:Proj}.

\begin{remark}
\label{rem:plainconvOCP}
For the OCP's from Sections \ref{sec:OCPPoisson}--\ref{sec:OCPMaxwell},
with the discretizations proposed in Section \ref{sec:LSQFEM},
it has been shown in Remark \ref{rem:plainconvexamples}
that (S1)--(S3) and (L) hold for $\VL_Y = \VL_P$.
For the heat equation OCP from Section \ref{sec:OCPSpacetime},
with the discretization from Section \ref{sec:LSQFEM},
it follows that (S1)--(S3) and (L) hold for $\VL_Y$ and $\VL_P$,
because all steps in the proof of \cite[Proof of Theorem 3.3]{GaStLSQ}
also hold with the additional term $\norm[L^2(\spatialdomain)]{ u_1(T,\cdot) }^2$
in the definition of the norm on $\bbV_Y = \bbV_P$.
This term can be treated in the same way as the term
$\norm[L^2(\spatialdomain)]{ u_1(0,\cdot) }^2$,
which was already considered in \cite[Proof of Theorem 3.3]{GaStLSQ}.

It then directly follows from the fact that 
Assumptions (S1)--(S3) are satisfied for $\VL_Y$ and $\VL_P$,
that they are also satisfied for $\VL$ defined in \eqref{eq:FoSLSOCP}.
It remains to verify (L) for $\VL$.
Using that (L) has already been shown for $\VL_Y$ and $\VL_P$,
it remains to show that 
$\VB \Pi_\ad \VC^{-1} \VB^* \VI_P$
and 
$\VA^* \VA \VI_Y$
satisfy (L).
For all $K \in \calT$,
it follows directly from the definitions of the maps $\VA, \VA^*, \VB, \VB^*, \VC^{-1}, \Pi_\ad$ 
in the three examples from Sections \ref{sec:OCPPoisson}--\ref{sec:OCPSpacetime}
that
\begin{align*}
\norm[\widehat{\bbL}(K)]{ \VA V } \leq &\, C \norm[\widehat{\bbL}(K)]{ V },
&
\norm[\widehat{\bbL}(K)]{ \VA^* V } \leq &\, C \norm[\widehat{\bbL}(K)]{ V },
\\
\norm[\widehat{\bbL}(K)]{ \VB W } \leq &\, C \norm[\bbX(K)]{ W },
&
\norm[\bbX(K)]{ \VB^* V } \leq &\, C \norm[\widehat{\bbL}(K)]{ V },
\\
\norm[\bbX(K)]{ \VC^{-1} W } \leq &\, C \norm[\bbX(K)]{ W },
&
\norm[\bbX(K)]{ \Pi_\ad W } = &\, \norm[\bbX(K)]{ W }
.
\end{align*}
For $\VI_Y, \VI_P$ from Sections \ref{sec:OCPPoisson}--\ref{sec:OCPMaxwell},
the maps $\VI_Y : \bbV_Y \to \widehat{\bbL}, \VI_P : \bbV_P \to \widehat{\bbL}$ 
are continuous embeddings, hence
$\norm[\widehat{\bbL}(K)]{ \VI_Y Y } \leq C \norm[\bbV_Y(K)]{ Y }$
and
$\norm[\widehat{\bbL}(K)]{ \VI_P P } \leq C \norm[\bbV_P(K)]{ P }$.
For $\VI_Y$ from Section \ref{sec:OCPSpacetime},
the following estimate now follows easily, 
because of the term 
$\norm[L^2(\partial_T K)]{ u_1(T,\cdot)|_{\partial_T K} }^2$
in the definition of the $\bbV_Y$-norm.
\begin{align*}
\norm[\widehat{\bbL}(K)]{ \VI_Y Y }^2 
	= &\, \| u_1 \|^2_{L^2(K)} 
	+ \| \bsu_2 \|^2_{L^2(K)^d} 
	+ \norm[L^2(\partial_T K)]{ u_1(T,\cdot)|_{\partial_T K} }^2
	\\
	\leq &\, \| u_1 \|^2_{L^2(K)}
	+ \| \nabla_x u_1 \|^2_{L^2(K)^d} 
	+ \| \bsu_2 \|^2_{L^2(K)^d} 
	+ \| \divv \bsu \|^2_{L^2(K)}
	\\
	&\, + \norm[L^2(\partial_0 K)]{ u_1(0,\cdot)|_{\partial_0 K} }^2
	+ \norm[L^2(\partial_T K)]{ u_1(T,\cdot)|_{\partial_T K} }^2
	\\
	= &\, \norm[\bbV_Y(K)]{ Y }^2
	,
\end{align*}
where we denoted 
$Y = ( u_1, \bsu_2 )$.
Similarly,
it holds that
$\norm[\widehat{\bbL}(K)]{ \VI_P P } \leq \norm[\bbV_P(K)]{ P }$.
Thus, $\VA, \VA^*, \VB, \VB^*, \VC^{-1}, \Pi_\ad, \VI_Y, \VI_P$
are bounded linear operators on the function spaces on $K$,
which shows that the same holds for their compositions 
$\VB \Pi_\ad \VC^{-1} \VB^* \VI_P$
and 
$\VA^* \VA \VI_Y$.
This completes the proof of (L) for $\VL$.
\end{remark}

\begin{remark} \label{rmk:Proj}
By Remark \ref{rem:ocpa2}, minimization of $\LS(v;F)$ over $v\in\bbV$
is equivalent to minimization of $\LS(v;\Pi_{{\rm range}(\VL)} F)$ over $v\in\bbV$,
which satisfies (A1)--(A4).
By Remark \ref{rem:plainconvOCP},
application of Proposition \ref{prop:plainconv} shows that
the iterates of Algorithm \ref{algo:AFEM} converge and satisfy
$\norm[\bbV(\domain)]{ U_\ell - U } \to 0$ for $\ell \to \infty$
and
$\LS(U_\ell;\Pi_{{\rm range}(\VL)} F) \to 0$ for $\ell \to \infty$,
thus $ \LS(U_\ell;F) \to \| \Pi_{{\rm range}(\VL)} F - F \|^2_{\bbL}$
by \eqref{eq:LSQLdecomp}.
\end{remark}

With the FoSLS NNs from Section \ref{sec:deepFoSLS},
Proposition \ref{prop:plainconv} 
together with Remarks \ref{rem:plainconvexamples}, \ref{rem:plainconvOCP}
and \ref{rmk:Proj}
directly implies convergence of NN approximations.

The above adaptive NN growth algorithms are known to convergence 
for some of the boundary value problems in FoSLS formulations
that were presented above, with optimal rates. 
We refer to \cite{CCPark2015,CollMarkLSQOptRat}
and the survey \cite{PBringmann23} and the references there.
%

%
\section{Conclusions and Perspectives}
\label{sec:Concl}
%
The proposed FoSLS formulations yield, for a wide range of elliptic
and parabolic PDEs in bounded, polyhedral domains,
\emph{variationally correct, numerically computable loss functions} 
for NN training.
The numerically computable FoSLS loss $\calE_{\mathrm{FOSLS}}(\theta)$ 
of a NN realization $U_\theta$
occurring at an admissible $\theta\in \Theta$ during 
training is a \emph{computable upper bound of the
approximation accuracy of $U_\theta$ with respect to the FoSLS solution $U$,
in the physically relevant norm}.
We indicated the connection of adaptive least squares FE algorithms
based on FoSLS formulations with $L^2(\domain)$ loss functions, with
\emph{Neural Network Growth} strategies.
Adaptive FoSLS based FEM, as emulated in the present paper
are known to be rate-optimal in most of the problems considered
in Sec.~\ref{sec:LSQForm}.
See, e.g., \cite{CCPark2015,CollMarkLSQOptRat}.
Proven mathematical statements of rate optimality on the sequence of
approximate solutions produced by adaptive least squares FE methods
transfer to corresponding adaptive NN growth strategies based on the
proposed LSQ formulations.
Within the proposed, adaptive LSQ formulations,
the ESTIMATE-MARK loop corresponds to
local equilibrium indicators for out-of-equilibrium of certain subnetworks.
Furthermore, algorithmic steering of mesh-refinement patterns
could also be delegated to NNs \cite{BohnMF21}.

We related the LSQ functional of FoSLS PDE formulations,
which is known to allow reliable and efficient error control,
to corresponding loss functions in NN training.
Similar strategies are conceivable for other variational formulations,
such as e.g. the ``deep Ritz'' approach of \cite{EYuDeepRitz}.
Rather than training approximating NNs by minimizing the quadratic potential energy,
NN training could be based 
on reliable and efficient residual a posteriori error estimators.

The proposed deep FoSLS LSQ formulation
extends to high-dimensional, parametric PDEs,
as arise in Uncertainty Quantification.
There, a LSQ regression formulation
in the product space $L^2(\domain)\otimes L^2_\bbP(P)$
for a suitable probability space $(P,\cA,\bbP)$ on the parameter domain $P$
naturally implies variance optimal, maximum likelihood estimator
properties of the corresponding optimal NN approximations.
The corresponding parametric solution families are
minimizers of suitable FoSLS formulations over
NNs resolving parametric solutions families.
Relevant issues,
such as the parsimonious emulation of polynomial chaos
expansions of parametric solution families,
and
extension to nonlinear PDEs shall be discussed elsewhere.
\subsection*{Acknowledgement} \label{sec:Ackn}
P.P was supported by the Austrian Science Fund (FWF) [P37010].
{\small
\bibliography{bibliography}

\begin{thebibliography}{10}

\bibitem{AinsDong}
M.~Ainsworth and J.~Dong.
\newblock Galerkin neural networks: a framework for approximating variational
  equations with error control.
\newblock {\em SIAM J. Sci. Comput.}, 43(4):A2474--A2501, 2021.

\bibitem{ainsworth2024extended}
M.~Ainsworth and J.~Dong.
\newblock Extended {G}alerkin neural network approximation of singular
  variational problems with error control, 2024.
\newblock ArXiv: 2405.00815.

\bibitem{ABMM2016}
R.~Arora, A.~Basu, P.~Mianjy, and A.~Mukherjee.
\newblock Understanding deep neural networks with rectified linear units.
\newblock In {\em International Conference on Learning Representations}, 2018.
\newblock arXiv: 1611.01491.

\bibitem{AHS22_989}
R.~Aylwin, F.~Henriquez, and C.~Schwab.
\newblock {ReLU Neural Network Galerkin BEM}.
\newblock {\em Journ. Sci. Computing}, 95(2), 2023.

\bibitem{bachmayr2024variationallycorrectneuralresidual}
M.~Bachmayr, W.~Dahmen, and M.~Oster.
\newblock Variationally correct neural residual regression for parametric
  {PDEs}: On the viability of controlled accuracy, 2024.
\newblock ArXiv: 2405.20065.

\bibitem{BM2023}
M.~Bernkopf and J.~M. Melenk.
\newblock Optimal convergence rates in {$L^2$} for a first order system least
  squares finite element method---{P}art {I}: {H}omogeneous boundary
  conditions.
\newblock {\em ESAIM Math. Model. Numer. Anal.}, 57(1):107--141, 2023.

\bibitem{BB2023}
F.~M. Bersetche and J.~P. Borthagaray.
\newblock A deep first-order system least squares method for solving elliptic
  {PDEs}.
\newblock {\em Computers \& Mathematics with Applications}, 129:136--150, 2023.

\bibitem{BoGuLSQ}
P.~B. Bochev and M.~D. Gunzburger.
\newblock {\em Least-squares finite element methods}, volume 166 of {\em
  Applied Mathematical Sciences}.
\newblock Springer, New York, 2009.

\bibitem{BohnMF21}
J.~Bohn and M.~Feischl.
\newblock Recurrent neural networks as optimal mesh refinement strategies.
\newblock {\em Comput. Math. Appl.}, 97:61--76, 2021.

\bibitem{PBringmann23}
P.~Bringmann.
\newblock How to prove optimal convergence rates for adaptive least-squares
  finite element methods.
\newblock {\em J. Numer. Math.}, 31(1):43--58, 2023.

\bibitem{CCLL2020}
Z.~Cai, J.~Chen, M.~Liu, and X.~Liu.
\newblock Deep least-squares methods: An unsupervised learning-based numerical
  method for solving elliptic pdes.
\newblock {\em Journal of Computational Physics}, 420:109707, 2020.

\bibitem{cai2023leastsquares}
Z.~Cai, J.~Choi, and M.~Liu.
\newblock Least-squares neural network {(LSNN)} method for linear
  advection-reaction equation: Non-constant jumps, 2023.
\newblock ArXiv: 2306.07445.

\bibitem{CKS2005}
Z.~Cai, J.~Korsawe, and G.~Starke.
\newblock An adaptive least squares mixed finite element method for the
  stress-displacement formulation of linear elasticity.
\newblock {\em Numer. Methods Partial Differential Equations}, 21(1):132--148,
  2005.

\bibitem{CollMarkLSQOptRat}
C.~Carstensen.
\newblock Collective marking for adaptive least-squares finite element methods
  with optimal rates.
\newblock {\em Math. Comp.}, 89(321):89--103, 2020.

\bibitem{CCPark2015}
C.~Carstensen and E.-J. Park.
\newblock Convergence and optimality of adaptive least squares finite element
  methods.
\newblock {\em SIAM J. Numer. Anal.}, 53(1):43--62, 2015.

\bibitem{CS2018}
C.~Carstensen and J.~Storn.
\newblock Asymptotic exactness of the least-squares finite element residual.
\newblock {\em SIAM Journal on Numerical Analysis}, 56(4):2008--2028, 2018.

\bibitem{CostabelCoerCMaxw}
M.~Costabel.
\newblock A coercive bilinear form for {M}axwell's equations.
\newblock {\em J. Math. Anal. Appl.}, 157(2):527--541, 1991.

\bibitem{dang2024adaptivegrowingrandomizedneural}
H.~Dang and F.~Wang.
\newblock Adaptive growing randomized neural networks for solving partial
  differential equations, 2024.
\newblock ArXiv: 2408.17225.

\bibitem{EYuDeepRitz}
W.~E and B.~Yu.
\newblock The deep {R}itz method: a deep learning-based numerical algorithm for
  solving variational problems.
\newblock {\em Commun. Math. Stat.}, 6(1):1--12, 2018.

\bibitem{EG2004}
A.~Ern and J.-L. Guermond.
\newblock {\em Theory and practice of finite elements}, volume 159 of {\em
  Applied Mathematical Sciences}.
\newblock Springer, New York, NY, 2004.

\bibitem{ErnGuermondBookI2021}
A.~Ern and J.-L. Guermond.
\newblock {\em Finite elements {I}---{A}pproximation and interpolation},
  volume~72 of {\em Texts in Applied Mathematics}.
\newblock Springer, Cham, 2021.

\bibitem{TFMKOptCtrl2022}
T.~F\"{u}hrer and M.~Karkulik.
\newblock Least-squares finite elements for distributed optimal control
  problems.
\newblock {\em Numer. Math.}, 154(3-4):409--442, 2023.

\bibitem{FuePraeLSQ}
T.~F\"{u}hrer and D.~Praetorius.
\newblock A short note on plain convergence of adaptive least-squares finite
  element methods.
\newblock {\em Comput. Math. Appl.}, 80(6):1619--1632, 2020.

\bibitem{FK23}
T.~Führer, R.~González, and M.~Karkulik.
\newblock Well-posedness of first-order acoustic wave equations and space-time
  finite element approximation, 2023.
\newblock ArXiv:2311.10536.

\bibitem{GaStLSQ}
G.~Gantner and R.~Stevenson.
\newblock Further results on a space-time {FOSLS} formulation of parabolic
  {PDE}s.
\newblock {\em ESAIM Math. Model. Numer. Anal.}, 55(1):283--299, 2021.

\bibitem{GaStLSQApplic}
G.~Gantner and R.~Stevenson.
\newblock Applications of a space-time {FOSLS} formulation for parabolic
  {PDEs}.
\newblock {\em IMA Journal of Numerical Analysis}, 2023.

\bibitem{Grisvard}
P.~Grisvard.
\newblock {\em Elliptic problems in nonsmooth domains}, volume~69 of {\em
  Classics in Applied Mathematics}.
\newblock Society for Industrial and Applied Mathematics (SIAM), Philadelphia,
  PA, 2011.
\newblock Reprint of the 1985 original.

\bibitem{HLXZ2020}
J.~He, L.~Li, J.~Xu, and C.~Zheng.
\newblock {ReLU} deep neural networks and linear finite elements.
\newblock {\em J. Comp. Math.}, 38, 2020.

\bibitem{HX2023}
J.~He and J.~Xu.
\newblock Deep neural networks and finite elements of any order on arbitrary
  dimensions, 2024.
\newblock ArXiv:2312.14276.

\bibitem{JLLOptCtrl}
J.-L. Lions.
\newblock {\em Optimal control of systems governed by partial differential
  equations}.
\newblock Die Grundlehren der mathematischen Wissenschaften, Band 170.
  Springer-Verlag, New York-Berlin, 1971.
\newblock Translated from the French by S. K. Mitter.

\bibitem{LODSZ22_991}
M.~Longo, J.~A.~A. Opschoor, N.~Disch, C.~Schwab, and J.~Zech.
\newblock De {Rham} compatible deep neural network {FEM}.
\newblock {\em Neural Networks}, 165:721--739, 2023.

\bibitem{LZCC2022}
L.~Lyu, Z.~Zhang, M.~Chen, and J.~Chen.
\newblock Mim: A deep mixed residual method for solving high-order partial
  differential equations.
\newblock {\em Journal of Computational Physics}, 452:110930, 2022.

\bibitem{OS2024}
J.~A.~A. Opschoor and C.~Schwab.
\newblock Exponential expressivity of {ReLU$^k$} neural networks on {G}evrey
  classes with point singularities.
\newblock {\em Applications of Mathematics}, 2024.

\bibitem{TJRVarMio24}
D.~Patel, D.~Ray, M.~R.~A. Abdelmalik, T.~J.~R. Hughes, and A.~A. Oberai.
\newblock Variationally mimetic operator networks.
\newblock {\em Comput. Methods Appl. Mech. Engrg.}, 419:Paper No. 116536, 30,
  2024.

\bibitem{PV2018}
P.~Petersen and F.~Voigtlaender.
\newblock Optimal approximation of piecewise smooth functions using deep {ReLU}
  neural networks.
\newblock {\em Neural Netw.}, 108:296 -- 330, 2018.

\bibitem{DM2024}
T.~D. Ryck and S.~Mishra.
\newblock Numerical analysis of physics-informed neural networks and related
  models in physics-informed machine learning, 2024.
\newblock ArXiv:2402.10926.

\bibitem{Siebert2011}
K.~G. Siebert.
\newblock {A convergence proof for adaptive finite elements without lower
  bound}.
\newblock {\em IMA Journal of Numerical Analysis}, 31(3):947--970, 2011.

\bibitem{zeinhofer2024unifiedframeworkerroranalysis}
M.~Zeinhofer, R.~Masri, and K.-A. Mardal.
\newblock A unified framework for the error analysis of physics-informed neural
  networks, 2024.
\newblock (to appear in IMA Journ. Num. Analysis (2024)).

\end{thebibliography}
}

\end{document}